\documentclass[11pt]{amsart} 
\usepackage{amsmath,amssymb,latexsym,amsthm,tikz, enumerate,amscd,hyperref} \usepackage[abbrev]{amsrefs} 
\usepackage{graphicx}
\usepackage{todonotes}
\usepackage[margin=1in]{geometry}
\newtheoremstyle{dotless}{}{}{\itshape}{}{\bfseries}{}{ }{}
\usetikzlibrary{decorations.markings}

\newtheorem{Theorem}{Theorem}[section]

\newtheorem{Proposition}[Theorem]{Proposition} 
\newtheorem{corollary}[Theorem]{Corollary} 
 
\newtheorem{lemma}[Theorem]{Lemma}

\newtheorem*{theorem*}{Theorem}

\theoremstyle{definition} 
\newtheorem{definition}[Theorem]{Definition} 
\newtheorem*{Acknowledgements}{Acknowledgements}
\newtheorem{remark}[Theorem]{Remark}

\newtheorem{example}[Theorem]{Example}

\newtheorem*{question}{Question}

\DeclareMathOperator{\Cone}{\text{Cone}}
\DeclareMathOperator{\piprod}{\raisebox{-0.1em}{\huge{$\pi$}}\kern -0.2em}

\newcommand{\ce}{\mathcal {E}}

\newcommand{\cu}{\mathcal {U}}

\newcommand{\cz}{\mathcal {Z}}

\newcommand{\rr}{{\mathbb R}}

\def\clap#1{\hbox to 0pt{\hss#1\hss}}

\newcommand{\comment}[1]{} 
 
\newcommand{\Homeo}{\operatorname{Homeo}}
\newcommand{\diam}{\operatorname{diam}}

	{ 
\end{enumerate}
} 

	{ 
\end{enumerate}
} 

\newenvironment{enumeratei'}{ 
\begin{enumerate}[\upshape (i)$'$]}
	{ 
\end{enumerate}
} 

\newenvironment{enumerate1'}{ 
\begin{enumerate}[\upshape (1)$'$]}
	{ 
\end{enumerate}
} 

\newenvironment{enumeratea'}{ 
\begin{enumerate}[\upshape (a)$'$]}{ 
\end{enumerate}
}

\usepackage{color}
\usepackage[outerbars,color]{changebar}
\usepackage{ifthen}

\newboolean{usecolor}
\setboolean{usecolor}{true}
\ifthenelse{\boolean{usecolor}}
{
 \definecolor{colore}{cmyk}{0,1,0.6,0}
 \definecolor{coloregen}{cmyk}{0.7,0,1,0}
 \definecolor{coloresimo}{cmyk}{1,0.6,0,0}
}
{
 \definecolor{colore}{cmyk}{0,0,0,1}
 \definecolor{coloregen}{cmyk}{0,0,0,1}
 \definecolor{coloresimo}{cmyk}{0,0,0,1}
}

\numberwithin{equation}{section} 
\begin{document}
\title{Compressible Spaces and $\ce\cz$-Structures}
\author{Craig Guilbault}
\address{Department of Mathematical Sciences\\
University of Wisconsin-Milwaukee, Milwaukee, WI 53201}
\email{craigg@uwm.edu}
\author{Molly Moran}
\address{Department of Mathematics, The Colorado College, Colorado Springs, Colorado 80903}
\email{molly.moran@coloradocollege.edu}
\author{Kevin Schreve}
\address{Department of Mathematics, University of Chicago, Chicago, IL 60637}
\email{kschreve@math.uchicago.edu}
\date{}
\keywords{$\cz$-structure, boundary, graph of groups $\ce \cz$-structure}

\begin{abstract}
\noindent Bestvina introduced a $\mathcal{Z}$-structure for a group $G$ to
generalize the boundary of a CAT(0) or hyperbolic group. A refinement of this
notion, introduced by Farrell and Lafont, includes a $G$-equivariance
requirement, and is known as an $\mathcal{E}\mathcal{Z}$-structure. A recent
result of the first two authors with Tirel put $\mathcal{E}\mathcal{Z}%
$-structures on Baumslag-Solitar groups and $\mathcal{Z}$-structures on
generalized Baumslag-Solitar groups. We generalize this to higher dimensions
by showing that fundamental groups of graphs of closed nonpositively curved
Riemannian $n$-manifolds (each vertex and edge manifold is of dimension $n$)
admit $\mathcal{Z}$-structures and graphs of negatively curved or flat
Riemannian $n$-manifolds admit $\mathcal{E}\mathcal{Z}$-structures. \bigskip

\noindent\textbf{AMS classification numbers}. Primary: 20F36, 20F55, 20F65,
57S30, 57Q35, Secondary: 20J06, 32S22

\end{abstract}
\maketitle

\section{Introduction}

Suppose that a discrete group $G$ acts properly and cocompactly by isometries
on a CAT(0) space $X$. The space $X$ can be naturally compactified by
attaching the visual boundary $\partial_{\infty}X$, and $\bar X = X
\cup\partial_{\infty}X$ is homotopy equivalent to $X$. In fact, $\partial
_{\infty}X$ is a \emph{$\cz$-set} in $\bar X$, meaning that $\partial_{\infty
}X$ can be instantly homotoped\footnote{More precisely, there is a homotopy
$h_{t}: \bar X \rightarrow\bar X$ so that $h_{0} = \text{Id}_{\bar X}$ and
$h_{t}(\bar X) \subset X$ for all $t > 0$.} into $X$. Furthermore, large
translates of compact sets in $X$ ``look small", i.e. if $K$ is a compact set
and $x_{0} \in X$, then the visual diameter of $\{g_{n}K\}_{n = 1}^{\infty}$
as viewed at $x_{0}$ goes to $0$ if $d(1,g_{n}) \rightarrow\infty$ in $G$. It
follows that for any open cover $\cu$ of $\bar X$, all but finitely many
translates of $K$ are contained in an element $U \in\cu$.

In \cite{bestvina}, Bestvina introduced the concept of a $\cz$--structure on a
group $G$ in order to generalize the boundary of a CAT(0) (or word hyperbolic) group:

\begin{definition}
A \emph{$\cz$-structure} on a group $G$ is a pair of spaces $(\bar X,Z)$
satisfying the following four conditions:

\begin{enumerate}
\item $\bar{X}$ is a compact absolute retract.

\item $Z$ is a $\cz$-set in $\bar X$,

\item $X = \bar X - Z$ is a proper metric space on which $G$ acts properly and
cocompactly by isometries,

\item $X$ satisfies the \emph{nullity condition} with respect to the
$G$-action: for every compact $K \subset X$ and any open cover $\cu$ of $X$,
all but finitely many $G$-translates of $K$ lie in an element of $\cu$.
\end{enumerate}
\end{definition}

If the group action of $G$ on $X$ extends to $\bar{X}$, this is called an
$\ce\cz$-structure \cite{fl}. In Bestvina's original definition, the action on
$X$ was required to be free and $\bar{X}$ was required to be a Euclidean
retract (finite-dimensional absolute retract). The modified definition here is
due to Dranishnikov \cite{dr}. Among other things, this modification allows
for groups with torsion. It is still open whether all groups of type $F$ (or
even type $F^{\ast}$ or $F_{AR}^{\ast}$) admit $\cz$-structures\footnote{See
\cite{gm} for definitions.}.

Many properties of CAT(0) or hyperbolic boundaries transfer over to this
general setting. For example: the dimension of a $\cz$-boundary is equal to
the global cohomological dimension of that boundary; the \v{C}ech cohomology
of $Z$ determines $H^{\ast}\left(  G;\mathbb{Z}G\right)  $; and, when $G$ is
torsion-free, the cohomological dimension of $G$ is $1+\dim Z$. See
\cite{bestvina}, \cite{dr} and \cite{gm}.

In this paper, we are concerned with the following question:

\begin{question}
Suppose a group $G$ acts properly and cocompactly on a CAT(0) space $X$ by \emph{homeomorphisms}. Under what conditions is $(\bar X, \partial_\infty X)$ part of a $(\ce)\cz$-structure for $G$?
\end{question}

We note that the change to an action by homeomorphisms does not affect
property $(2)$ of Definition 1.1. In fact, in \cite{amn}, it is shown that in
this case there is a topologically equivalent metric on $X$ for which the
action is by isometries. Thus, in this situation, properties $(1)-(3)$ are
always satisfied. The key property to check then is the nullity condition. The
nullity condition can certainly fail for certain actions, for example, the
Baumslag-Solitar group
\[
BS(1,2) = \langle a,t | tat^{-1} = a^{2} \rangle
\]
acts properly and cocompactly on the product $T \times\rr$, where $T$ is a
trivalent tree. Certain translates of a fundamental domain have exponentially
growing height (as measured by the distance to a fixed basepoint) in $T
\times\rr$, and this implies $\bar X$ fails the nullity condition.

To get around this problem, the first two authors with Tirel in \cite{gmt}
modified the action of $BS(1,2)$ (or more generally $BS(m,n)$) on $T\times\rr$
by \textquotedblleft compressing the $\rr$-direction". The rough idea is to
conjugate the given action by a homeomorphism of $T\times\rr$ which is the
identity on $T$ and shrinks distances in the $\rr$-direction. This changes the
exponential growth of the diameter of translates of a fundamental domain to
sublinear growth (the diameters are still forced to go to infinity), which is
enough to ensure the nullity condition.

In this paper, we give a more general picture of when this strategy works. Our
setup requires that the CAT(0) space splits as a Cartesian product $X\times
Y$. Then, given a fundamental domain $K$, we insist that all $G$-translates of
$K$ have uniformly bounded diameter in the $X$-direction, and have
\textquotedblleft properly controlled\textquotedblright\ diameter in the
$Y$-direction. Finally, we assume $Y$ is \emph{compressible}, which roughly
means that $Y$ admits homeomorphisms which uniformly shrink all compact sets
$K\subset Y$ (see Definition \ref{d:compressible}). With this setup, we can
perform the same conjugating trick as in \cite{gmt} to produce a $\cz$-structure.

The compressibility hypothesis on $Y$ is rather strong since it requires that
compressing be done via homeomorphisms. This rules out a variety of important
candidates for the space $Y$. Remark \ref{Remark: radial compressions} and
Section \ref{Section: Noncompressible spaces} address some of the limitations
imposed by this hypothesis. Nevertheless, we show that all simply connected
nonpositively curved Riemannian manifolds are compressible. This allows us to
take on an important and well-studied class of groups. One of our main
theorems is the following.

\begin{Theorem}\label{t:maini}
Fundamental groups of graphs of closed nonpositively curved Riemannian $n$-manifolds have $\mathcal{Z}$-structures.
\end{Theorem}

Here, a \emph{graph of closed nonpositively curved $n$-manifold groups} means
a finite connected graph of groups, where each vertex and edge group is the
fundamental group of a closed nonpositively curved Riemannian $n$-manifold
(where $n$ is the same for each vertex and edge). Equivalently, we could have
required that each vertex group is the fundamental group of a closed
nonpositively curved $n$-manifold and each edge group is finite index in its
vertex groups. The coarse geometric structure of these groups has been studied
by Farb-Mosher in \cite{fm}, and more generally by Mosher-Sageev-Whyte in
\cite{msw}. In those papers, graphs of groups with the finite index property
for edge groups are referred to as \emph{geometrically homogeneous.}

The beauty of the geometric homogeneity property is that the fundamental group
frequently acts properly and cocompactly on a Cartesian product space
$T\times\widetilde{M}_{v}$, where $\widetilde{M}_{v}$ is the universal cover
of a vertex space (in our case and in the case of \cite{fm}, a Riemannian
manifold homeomorphic to $%
%TCIMACRO{\U{211d} }%
%BeginExpansion
\mathbb{R}
%EndExpansion
^{n}$) and $T$ is the associated Bass-Serre tree. Even though the action is
usually not by isometries, the CAT(0) geometry of $T\times\widetilde{M}_{v}$
under the $\ell_{2}$-metric, with its corresponding visual boundary, turn out
to be quite useful.

In addition to the above theorem, we give conditions under which our (new)
$G$-action extends to $\bar{X}$. Roughly speaking, we need to know that
certain covering maps lift to maps on the universal cover which extend to
visual boundaries; we also need a compressing homeomorphism that does not
change much under linear reparametrization of the domain. If each vertex
manifold is negatively curved or flat, our conditions will be satisfied, so we
have the following:

\begin{Theorem}
Fundamental groups of graphs of closed negatively curved or flat Riemannian $n$-manifolds have $\ce \mathcal{Z}$-structures.
\end{Theorem}

As in the proof of Theorem \ref{t:maini} the CAT(0) geometry of $T\times\widetilde{M}_{v}$ plays a
vital role in the proof of this theorem.

A somewhat surprising ingredient in some part of the proofs of our main theorems is the use
of the theory of Hilbert cube manifolds as pioneered by Anderson,
Chapman, West,  Toru\'{n}czyk, Edwards, and others. See
\cite{Cha75} and \cite{vMi89} for overviews of the subject. Most---but not
all---of our results can be obtained without the use of Hilbert cube
manifolds, the key exception being the $n=4$ cases of our main theorems.
Perhaps more importantly, the Hilbert cube manifolds approach (which is valid
in all dimensions) has the potential for proving theorems beyond the scope of
this paper. We expand upon that thought in Section
\ref{Section: Some open questions}. The authors thank the referee who pushed
us to include this approach in the current article. Prior to the requested
revisions, we had overlooked a gap in our proofs when $n=4$.

\begin{Acknowledgements}
This research was supported in part by Simons Foundation Grant 427244, CRG. This material is based upon work done while the third author was supported by the National Science Foundation under Award No. 1704364. In addition to suggesting the Hilbert cube manifold approach, we thank the referee for carefully reading the article and other useful comments. 
\end{Acknowledgements}

\section{Boundaries of CAT(0) spaces}

We will assume the reader is familiar with the definitions and basic facts
about CAT(0) spaces, see \cite{bh} for complete details.

Let $X$ be a proper CAT(0) space, and let $\partial X$ be the boundary of $X$.
Fix a base point $x_{0}\in X$. Each equivalence class of rays in $\partial X$
contains exactly one representative emanating from $x_{0}$. We may endow
$\overline{X}=X\cup\partial X$, with the \emph{cone topology}, described
below, under which $\partial X$ is a closed subspace of $\overline{X}$ and
$\overline{X}$ compact. Equipped with the topology induced by the cone
topology on $\overline{X}$, the boundary is called the \emph{visual boundary}
of $X$; we denote it by $\partial_{\infty}X$.

The cone topology on $\overline{X}$, for $x_{0}\in X$, is generated by the
basis $\mathcal{B}=\mathcal{B}_{0}\cup\mathcal{B}_{\infty}$ where
$\mathcal{B}_{0}$ consists of all open balls $B(x,r)\subset X$ and
$\mathcal{B}_{\infty}$ is the collection of all sets of the form
\[
U(c,r,\epsilon)=\{x\in\overline{X}\mid d(x,x_{0})>r\text{ and }d(p_{r}%
(x),c(r))<\epsilon\}
\]
where $c:[0,\infty)\rightarrow X$ is any geodesic ray based at $x_{0}$, $r>0$,
$\epsilon>0$, and $p_{r}$ is the natural projection of $\overline{X}$ onto
$\overline{B}(x_{0},r)$.

Now, for each pair $x\in X$ and $\varepsilon>0$, let
\[
V( x,\varepsilon) =\{y\in\overline{X}\mid d( x_{0},y) >d( x_{0},x) \text{ and
}d( x,p_{d( x_{0},x) }( y) ) <\varepsilon\}
\]

\begin{lemma}
If $\mathcal{B}_{0}$ is the set of all open balls in a proper CAT(0) space $X$
and $\mathcal{V}_{x_{0}}$ is the collection of all $V( x,\varepsilon) $ as
defined above, then $\mathcal{B}_{0}\cup\mathcal{V}_{x_{0}}$ is a basis for
the usual cone topology on $\overline{X}$.
\end{lemma}

\begin{proof}
Clearly each set $U( \xi,r,\varepsilon) $ can be expressed as $V( \xi( r)
,\varepsilon) $; so the cone topology is at least as fine as the proposed
topology. For the reverse containment, suppose $y\in V( x,\varepsilon)
\in\mathcal{V}_{x_{0}}$. Let $\delta=\varepsilon-d( x,p_{d( x_{0},x) }( y) )$.
If $y\in\partial X$ then $y\in U( y,d( x_{0},x),\delta) \subseteq V(
x,\varepsilon)$. If $y\in X$ let $W=B( y,\delta) \backslash\overline{B(
x_{0},d(x_{0},x) ) }$. Since projection onto $\overline{B(x_{0},d( x_{0},x) )
}$ does not increase distances, $W\subseteq V( x,\varepsilon) $. It follows
that $V(x,\varepsilon) $ is open in the cone topology, so the proposed
topology is at least as fine as the cone topology.
\end{proof}

The following lemma is similar in spirit to the Lebesgue covering lemma and is
a generalization of Lemma 2.3 in \cite{gmt}.

\begin{lemma}
\label{Lemma: Cover of Boundary} Let $(X,d)$ be a proper CAT(0) space and let
$\mathcal{U}$ be an open cover of $\overline{X}$. Then there exists $R\gg0$
and $\delta>0$ so that for every $x\in X\backslash B( x_{0},R) $,
$V(x,\delta)$ lies in an element of $\mathcal{U}$.
\end{lemma}

\begin{proof}
Without loss of generality we may assume $\mathcal{U}$ consists entirely of
elements from the basis $\mathcal{B}_{0} \cup\mathcal{V}_{x_{0}}$. Since
$\partial_{\infty}X$ is compact, there exists $\{U_{1},U_{2},...,U_{k}%
\}\subseteq\mathcal{U}$ that covers $\partial_{\infty}X$. For each $i$, write
$U_{i}=V(x_{i},\varepsilon_{i}) $. Since $X\backslash\cup_{i-1}^{k}U_{i}$ is a
closed subset of $X$ which contains no infinite rays, an Arzela-Ascoli
argument shows that $X\backslash\cup_{i-1}^{k}U_{i}$ is bounded. Choose
$R\gg0$ so that
\[
R>\max\{ d( x_{0},x_{i}) \mid1\leq i\leq k\} \text{ and } X\backslash
\cup_{i-1}^{k}U_{i}\subseteq B( x_{0},R).
\]
Note that if an open ball $B(x,\varepsilon)$ lies in $U_{i}$ then
$V(x,\varepsilon) \subseteq U_{i}$. It follows that, for each $x\in S(
x_{0},R) $ there exists some $\varepsilon_{x}>0$ so that $V( x,\varepsilon
_{x}) $ is contained in some $U_{i}$. For each $i\in\{1,2,...,k\}$, define a
function $\eta_{i}:S( x_{0},R) \rightarrow\lbrack0,\infty)$ by $\eta
_{i}(x)=\text{sup}\{\epsilon\,|\,V(x,\epsilon)\subseteq V_{i}\}$. Note that
$\eta_{i}$ is continuous and $\eta_{i}(x)>0$ if and only if $x\in V_{i}$.
Thus, $\eta:S( x_{0},R) \rightarrow\lbrack0,\infty)$ defined by $\eta
(x)=\text{max} \{\eta_{i}(x)\}_{i=1}^{k}$ is continuous and strictly positive.
Let $\delta^{\prime}$ be the minimum value of $\eta$ and set $\delta
=\min\{\frac{\delta^{\prime}}{2},\frac{1}{R}\}$. Clearly $V( x,\delta) $ lies
in some $U_{i}$ for all $x\in S( x_{0},R) $. Moreover, if $d( x_{0},x) >R$
then $V( x,\delta) \subseteq V( p_{R}( x) ,\delta) $; so again $V( x,\delta) $
lies in some $U_{i}$.
\end{proof}

\begin{definition}
A function $\phi: \rr^{+} \rightarrow\rr^{+}$ is \emph{sublinear} if
\[
\lim_{x \rightarrow\infty} \frac{\phi(x)}{x} = 0.
\]
A function $\phi: \rr^{+} \rightarrow\rr^{+}$ is \emph{uniformly sublinear}
if
\[
\bar\phi(t) := \max_{|x-y| = t} |\phi(x) - \phi(y)|
\]
is sublinear.
\end{definition}

For example, $\log(x + 1): \rr^{+} \rightarrow\rr^{+}$ is a uniformly
sublinear homeomorphism.

\begin{lemma}
\label{l:cat0angle} Suppose $(X,d)$ is a proper CAT(0) space. If $\phi:\rr^{+}
\rightarrow\rr^{+}$ is sublinear and $\mathcal{U}$ is an open cover of
$\overline{X}$ then there exists $T>0$ so that whenever $d( x_{0},x) >T$, $B(
x,\phi( d( x_{0},x) ) ) $ lies in some $U\in\mathcal{U}$.
\end{lemma}

\begin{proof}
Choose $R\gg0$ and $\delta>0$ as in the previous lemma. We can assume that
$\delta< \frac{1}{R}$. By sublinearity, choose $T>0$ so that $\frac{\phi( t)
}{t-\phi( t) }<\delta^{2}$ and $t-\phi( t) >R$ for all $t\geq T$. It suffices
to prove:\medskip

\noindent\textsc{Claim. }\emph{If} $d( x_{0},x) >T$\emph{, then }$B(x,\phi(d(
x_{0},x))) \subseteq V(p_{R}(x),\delta)$\emph{.}\medskip

Let $y\in B(x,\phi(d(x_{0},x)))$. We have that $d(x,y))<\phi(d(x_{0},x))$ and
$d(x_{0},x)-\phi(d(x_{0},x))>R$. Let $x^{\prime}=p_{d(x_{0},x)-\phi
(d(x_{0},x))}(x)$ and $y^{\prime}=p_{d(x_{0},x)-\phi(d(x_{0},x))}(y)$. Since
projection does not increase distances, $d(x^{\prime},y^{\prime})<\phi
(d(x_{0},x))$. Then
\begin{align*}
\frac{d(p_{R}(x^{\prime}),p_{R}(y^{\prime}))}{1/\delta} &  \leq\frac
{d(p_{R}(x^{\prime}),p_{R}(y^{\prime}))}{R}\\
&  \leq\frac{d(x^{\prime},y^{\prime})}{d(x_{0},x)-\phi(d(x_{0},x))}%
\text{\quad(CAT(0) inequality for }\triangle x_{0}x^{\prime}y^{\prime
}\text{.)}\\
&  \leq\frac{\phi(d(x_{0},x))}{d(x_{0},x)-\phi(d(x_{0},x))}\\
&  <\delta^{2}%
\end{align*}
It follows that $d(p_{R}(x^{\prime}),p_{R}(y^{\prime}))<\delta$; and since
$p_{R}(x^{\prime})=p_{R}(x)$ and $p_{R}(y^{\prime})=p_{R}(y)$, $y\in
V(p_{R}(x),\delta)$.
\end{proof}

%as $i \rightarrow \infty$. Therefore, for all but finitely many $i$, $d(p_{\frac{1}{\delta}}(x),z_i(\frac{1}{\delta}))<\delta$.

\section{The Hilbert cube and Hilbert cube
manifolds\label{Section: The Hilbert cube and Hilbert cube manifolds}}

To begin this section, we provide a brief discussion of a simple CAT(0) space
that will play a useful role in our main theorems. Let $I^{\omega}$ denote the
infinite product $%
%TCIMACRO{\tprod _{i=0}^{\infty}}%
%BeginExpansion
{\textstyle\prod_{i=0}^{\infty}}
%EndExpansion
\left[  0,\frac{1}{2^{i}}\right]  $ endowed with the metric $d\left(  \left(
x_{i}\right)  ,\left(  y_{i}\right)  \right)  =\left(
%TCIMACRO{\tsum }%
%BeginExpansion
{\textstyle\sum}
%EndExpansion
\left\vert x_{i}-y_{i}\right\vert ^{2}\right)  ^{1/2}$. This metric induces
the standard product topology, so $I^{\omega}$ is just a metrized version of
the \emph{Hilbert cube}. For each nonnegative integer $n$, endow $I^{n}=%
%TCIMACRO{\tprod _{i=0}^{n-1}}%
%BeginExpansion
{\textstyle\prod_{i=0}^{n-1}}
%EndExpansion
\left[  0,\frac{1}{2^{i}}\right]  \subseteq%
%TCIMACRO{\U{211d} }%
%BeginExpansion
\mathbb{R}
%EndExpansion
^{n}$ with the subspace metric, where $%
%TCIMACRO{\U{211d} }%
%BeginExpansion
\mathbb{R}
%EndExpansion
^{n}$ denotes $n$-dimensional Euclidean space. Clearly $I^{n}$ is CAT(0) and
the obvious inclusion $I^{n}\hookrightarrow I^{\omega}$ is an isometric
embedding. It follows that $I^{\omega}$ is CAT(0) since a \textquotedblleft
fat triangle\textquotedblright\ in $I^{\omega}$ could otherwise be projected
into $I^{n}$ (for large $n$) to obtain a fat triangle in $I^{n}$.

If $\left(  X,d\right)  $ is a proper CAT(0) space, then so is $X\times
I^{\omega}$, under the $\ell_{2}$ metric. Moreover, $\partial_{\infty}\left(
X\times I^{\omega}\right)  \approx\partial_{\infty}X$; in fact, for arbitrary
$p\in I^{\omega}$, $X\times\left\{  p\right\}  \hookrightarrow X\times
I^{\omega}$ extends continuously to an inclusion map $\widehat{X}%
\times\left\{  p\right\}  \hookrightarrow\widehat{X\times I^{\omega}}$ which
is a homeomorphism between visual boundaries.\medskip

A \emph{Hilbert cube manifold} is a separable metric space with the property
that every point has a neighborhood homeomorphic to $I^{\omega}$. A surprising
fact about $I^{\omega}$ is that it is homogeneous, i.e., for any $x,y\in
I^{\omega}$ there exists a homeomorphism of $I^{\omega}$ taking $x$ to $y$.
From there, it is easy to see that (as in finite dimensions) all connected
Hilbert cube manifolds are homogeneous.

In Section \ref{Section: Graphs of nonpositively curved n-manifolds}, we will
make essential use of some classical theorems from the topology of Hilbert
cube manifolds. We state those results here for easy access.

\begin{Theorem}
[\cite{Wes71}]If $X$ is a locally finite CW complex then $A\times I^{\omega}$
is a Hilbert cube manifold.
\end{Theorem}

\begin{Theorem}
[\cite{Cha74}]A map $f:A\rightarrow B$ between finite CW complexes is a simple
homotopy equivalence (in the sense of \cite{Whi50} and \cite{Coh73}) if and
only if \thinspace$f\times\operatorname*{id}_{I^{\omega}}:A\times I^{\omega
}\rightarrow B\times I^{\omega}$ is homotopic to a homeomorphism.
\end{Theorem}

Building upon his own work, along with work by West \cite{Wes77} and others,
Chapman \cite{Cha77} used Hilbert cube technology to extend simple homotopy
theory to the category of compact ANRs. From there, Edwards, building on work
by Toru\'{n}czyk \cite{Tor80}, \cite{Tor81}, obtained the following
generalization of the above theorems. Strictly speaking, we do not need this generalization for our main theorem, but believe it could be useful in attacking the problems posed in Section \ref{Section: Some open questions}.

\begin{Theorem}
The product of a locally compact ANR with $I^{\omega}$ is a Hilbert cube
manifold. Moreover, a map $f:A\rightarrow B$ between compact ANRs is a simple
homotopy equivalence if and only if \thinspace the map $f\times
\operatorname*{id}_{I^{\omega}}:A\times I^{\omega}\rightarrow B\times
I^{\omega}$ between Hilbert cube manifolds is homotopic to a homeomorphism.
\end{Theorem}

\section{$\mathcal{Z}$-structures and compressible spaces}

The following is our main definition. Recall that a function $f: X \rightarrow
Y$ is proper if preimages of compact sets are compact.

\begin{definition}
\label{d:compressible} A proper metric space $Y$ is \emph{compressible} if for
any proper function $\psi:\rr^{+}\rightarrow\rr^{+}$, there is a homeomorphism
$h_{\psi}:Y\rightarrow Y$ so that for every compact set $K\subset Y$ with
$\diam(K)<\psi(R),\diam(h_{\psi}(K))<\phi(R)$ for $\phi:\rr^{+}\rightarrow
\rr^{+}$ a sublinear function. We say $h_{\psi}$ is a \emph{compressing
homeomorphism}.
\end{definition}

Of course, it suffices to check compressibility for $K$ being two points, but
it is convenient to state it for all compact sets. For applications to
geometric group theory, the following is obvious, but useful.

\begin{Proposition}\label{p:qi}
If a map $f:X\rightarrow Y$ between proper metric spaces is both a
homeomorphism and a quasi-isometry, and $X$ is compressible, then so is $Y$.
\end{Proposition}

We can now state our main technical theorem, which generalizes the main result
of \cite{gmt}.

\begin{Theorem}\label{t:main}
Suppose $G$ acts properly and cocompactly on $X \times Y$, where $X$ and $Y$ are CAT(0) and $Y$ is compressible. Let
$(x_0, y_0)$ be a basepoint in $X \times Y$. Let $\pi_X: X \times Y \rightarrow X \times y_0$ and $\pi_Y: X \times Y \rightarrow x_0 \times Y$ be the projections onto each factor. Assume that for a fixed compact $K$ in $X \times Y$ with $GK = X \times Y$,
\begin{enumerate}
\item There is $S > 0$ so that $\diam(\pi_X(gK)) < S$ for all $g \in G$.
\item There is a proper function $\psi: \rr^+ \rightarrow \rr^+$
so that if $\pi_X(gK) \subset B_X(x_0, R)$, then $\diam \pi_Y(gK) < \psi(R)$.
\end{enumerate}
Then there exists a proper cocompact action of $G$ on $X \times Y$ so that $(X \times Y, \partial_\infty(X \times Y))$ is a $\mathcal{Z}$-structure for $G$.
\end{Theorem}

\begin{proof}
Let $h_{\psi}:Y\rightarrow Y$ be a compressing homeomorphism for $\psi$, and
extend this to the homeomorphism $H_{\psi}:=\left(  \operatorname*{id}%
_{X},h_{\psi}\right)  :X\times Y\rightarrow X\times Y$. Let $f:G\rightarrow
\Homeo(X\times Y)$ be the given action. Now, modify the action by conjugating
with $H_{\psi}$, i.e. define a new action by
\[
f^{H_{\psi}}(g):X\times Y\rightarrow X\times Y;\ f^{H_{\psi}}(g)=H_{\psi
}f(g)H_{\psi}^{-1}%
\]

This conjugated action is again proper and cocompact. We only need to verify
the nullity condition. If $K$ is a fundamental domain for the original
$G$-action, then $H_{\psi}(K)$ is a fundamental domain for the conjugated
action. Let $g\in G$, and consider a translate $f^{H_{\psi}}(g).K$. By
construction, this is contained in a product of balls $B(x,2S)\times
B(y,\phi(d(x_{0},x)))$ for sublinear $\phi$ and some $x \in X$. By Lemma
\ref{l:cat0angle}, the collection $\mathcal{B}=B(x,2S)\times B(y,\phi
(d(x_{0},x)))$ satisfies the nullity condition.
\end{proof}

We now give examples of compressible spaces.

\begin{lemma}
\label{r compressible} $\rr^{+}$ with the standard metric is compressible.
\end{lemma}

\begin{proof}
By choosing a larger function, we can assume the proper function $\psi
:\rr^{+}\rightarrow\rr^{+}$ is increasing and $\psi(x+y)\geq\psi(x)+\psi(y)$
for all $x,y\in\rr^{+}$. Since $\psi$ is increasing, we can define the inverse
$\psi^{-1}:[\psi(0),\infty)\rightarrow\rr^{+}$. Given an interval
$[a,a+\psi(R)]$ with $a>\psi(0)$, we have that:
\[
\psi^{-1}([a,a+\psi(R)]\subset\lbrack\psi^{-1}(a),\psi^{-1}(a)+R].
\]
Now, let the homeomorphism $\widehat{h}_{\psi}:\rr^{+}\rightarrow\rr^{+}$ be
defined by%

\[
\widehat{h}_{\psi}(x)=%
\begin{cases}
x & 0\leq x\leq\psi(0)\\
\phi(\psi^{-1}(x))+\psi(0) & x\geq\psi(0)\\
&
\end{cases}
\]
where $\phi$ is a uniformly sublinear homeomorphism with $\phi(0)=0$. It
follows that for any $a,a^{\prime}$ with $d(a,a^{\prime})<\psi(R)$,
$d(h_{\psi}(a),h_{\psi}(a^{\prime}))<\phi(R)+\psi(0)$.
\end{proof}

Note that we can also assume $\frac{\widehat{h}_{\psi}(x)}{x}$ is decreasing.
Our main examples of compressible spaces comes from the following theorem.

\begin{Theorem}\label{manifold compressible}
Let $M$ be a simply connected, nonpositively curved, Riemannian manifold.  Then $M$ is compressible.
\end{Theorem}

\begin{proof}
We assume the proper function $\psi:\rr^{+}\rightarrow\rr^{+}$ is increasing,
$\psi(x+y)\geq\psi(x)+\psi(y)$ for all $x,y\in\rr^{+}$, and $\widehat{h}%
_{\psi}$ is defined as in Lemma \ref{r compressible}. Fix a basepoint
$m_{0}\in M$. Let
\[
\exp_{m_{0}}:T_{m_{0}}M\cong\rr^{n}\rightarrow M
\]
be the exponential map which, by the Cartan-Hadamard Theorem, is a
diffeomorphism taking geodesic rays in $\rr^{n}$ emanating from the origin to
geodesic rays in $M$ emanating from $m_{0}$. Consider the homeomorphism
\begin{equation}
h_{\psi}:=\exp_{m_{0}}\circ h_{\psi}^{\prime}\circ\exp_{m_{0}}^{-1}%
:M\rightarrow M \label{Riemannian compression}%
\end{equation}
where $h_{\psi}^{\prime}$ restricts to $\widehat{h}_{\psi}$ on geodesic rays
in $\rr^{n}$ emanating from the origin. Roughly speaking, $h_{\psi}$ is the
homeomorphism that restricts to $\widehat{h}_{\psi}$ on geodesic rays in $M$
emanating from $m_{0}$. \medskip

\noindent\textbf{Claim. }$h_{\psi}$ \emph{is a compressing homeomorphism for
}$M$\emph{. }\medskip

Suppose that $x$ and $y$ are two points in $M$ with $d(x,y)<\psi(R)$. If
$d(x,m_{0})=d(y,m_{0})=D$, then $d(h_{\psi}(x),h_{\psi}(y))\leq\frac
{\widehat{h}_{\psi}(D)d(x,y)}{D}$ by the CAT(0) inequality. Since
$D>\frac{d(x,y)}{2}$, by our assumption that $\frac{\widehat{h}_{\psi}(x)}{x}$
is decreasing we have that:
\[
\frac{\widehat{h}_{\psi}(D)d(x,y)}{D}<\frac{\widehat{h}_{\psi}\left(
\frac{d(x,y)}{2}\right)  d(x,y)}{\frac{d(x,y)}{2}}=2\widehat{h}_{\psi}\left(
\frac{d(x,y)}{2}\right)  <2\phi(R)
\]

In general, assume that $d(x,m_{0})<d(y,m_{0})$, and choose $z$ so that
$d(x,m_{0})=d(z,m_{0})$ and $z$ lies on the same ray emanating from $m_{0}$ as
$y$. The projection of $x$ to the geodesic between $y$ and $m_{0}$ is less
than $d(x,m_{0})$ from $m_{0}$. By convexity of the distance function we have
$d(x,z)<d(x,y)$. Similarly, since metric balls are convex and $z$ is the
projection of $y$ onto the $d(x,m_{0})$-ball around $m_{0}$, we have
$d(y,z)<d(x,y)$. So, by assumption we have $d(x,z)$ and $d(z,y) < \psi(R)$. By
the above, we have that
\[
d(h_{\psi}(x),h_{\psi}(z))\leq2\widehat{h}_{\psi}(\psi(R))\leq2\phi(R)
\]
and by Lemma \ref{r compressible},
\[
d(h_{\psi}(z),h_{\psi}(y))<2\widehat{h}_{\psi}(\psi(R))\leq2\phi(R),
\]
so we have
\[
d(h_{\psi}(x),h_{\psi}(y))\leq d(h_{\psi}(x),h_{\psi}(y)) + d(h_{\psi
}(y),h_{\psi}(z)) \leq4\phi(R)
\]
and we are done.
\end{proof}

\comment{
\begin{example}
Let $(X,d_X)$ be a path metric space, and equip $CX= X \times \rr^+/X \times 0$. Let $\Phi: \rr^+ \rightarrow \rr^+$ be proper and monotone increasing. If $\gamma$ is a path in $CX$, define its length to be $$\sup\bigg\{ \sum_{j = 0}^{n-1} |t_j - t_{j+1}| + \max \{\Phi(t_j), \Phi(t_{j+1}\}d(x_j, x_{j+1})\}\bigg\}$$ where the supremum is taken over all finite sequences of points on $\gamma$.
Then define $d_{CX_\Phi}$ by taking the infimum of the lengths of paths between two points.
We claim $(CX, d_{CX_{\Phi}})$ is compressible. Let $\phi: \rr^+ \rightarrow \rr^+$ be proper, and let $h_\phi: \rr^+ \rightarrow \rr^+$ be a compressing homeomorphism. Let $H_\Phi: CX \rightarrow CX$
\end{example}
}

\begin{remark}
\label{Remark: radial compressions}A number of comments regarding compressible
spaces are in order.

\begin{enumerate}
\item For examples of compressible spaces not homeomorphic to $\mathbb{R}^{n}%
$, let $X$ be a compact metric space and $\operatorname*{Cone}_{\infty}\left(
X\right)  :=X\times\lbrack0,\infty)/X\times\left\{  0\right\}  $. There are
natural \textquotedblleft warped product" metrics that one can put on
$\operatorname*{Cone}_{\infty}\left(  X\right)  $ and, for suitable choices,
there are natural homeomorphisms which move points towards the cone point
along cone lines to produce a compressing homeomorphism. For example, $X$
could be $3$ points and $\Cone_{\infty}(X)$ the infinite tripod equipped with
the natural path metric.

\item Compressing homeomorphisms like the ones described above and in the
proof of Theorem \ref{manifold compressible} are called \emph{radial
compressions.} More specifically, if $Y$ is homeomorphic to an open cone and
$\widehat{h}:\mathbb{R}^{+}\rightarrow\mathbb{R}^{+}$ is a homeomorphism, then
the map $h:Y\rightarrow Y$ which acts as $\widehat{h}$ on each cone line is
called a \emph{radial homeomorphism. }When a CAT(0) space $Y$ admits an open
cone structure where the cone lines are geodesic rays emanating from a fixed
point $y_{0}\in Y$, the proof of Theorem \ref{manifold compressible} shows
that, for any proper function $\psi:\mathbb{R}^{+}\rightarrow\mathbb{R}^{+}$
there is a homeomorphism $\widehat{h}_{\psi}:\mathbb{R}^{+}\rightarrow
\mathbb{R}^{+}$ such that the corresponding radial homeomorphism $h_{\psi
}:Y\rightarrow Y$ is a compressing homeomorphism for $\psi$. We call such a
space \emph{radially compressible}.

\item \label{Item: compressible but not radially}For an example of a
compressible space that is not radially compressible, consider $\mathbb{R}%
^{n}\times X$, where $\mathbb{R}^{n}$ has the Euclidean metric and $X$ is
compact. More generally, the product of a compressible space with a compact
metric space will always be compressible, when given the $\ell_{2}$-metric.

\item For a noncompressible CAT(0) space, let $T$ be the universal cover of
the wedge of two circles. (For further discussion of this example, see Section
\ref{Section: Noncompressible spaces}).

\item We are particularly interested in compressibility of universal covers of
closed aspherical manifolds. As shown above, universal covers of nonpositively
curved closed Riemannian $n$-manifolds are always compressible. Our proof
generalizes to CAT(0) universal covers only when there are no points from which
geodesic rays bifurcate. In Section
\ref{Section: Noncompressible spaces} we will show that the exotic universal
covers constructed by Davis in \cite{Dav83}, many of which are CAT(0), are
noncompressible. Knowing that a universal cover is homeomorphic to
$\mathbb{R}^{n}$ does not appear to be enough; in fact, we suspect it is a
rare phenomenon that a proper metric on $\mathbb{R}^{n}$ yields a compressible space.

\item For a concrete example which exhibits our (lack of) knowledge outside
the nonpositively curved case, we do not know if the geometries NIL and SOL
are compressible, nor do we know if closed graph $3$-manifolds have
compressible universal cover.
\end{enumerate}
\end{remark}

\section{$\ce\mathcal{Z}$-structures\label{Section: EZ-structures}}

In this section we identify conditions on a compressing homeomorphism
$h_{\psi}$ and on the initial action of $G$ on $X\times Y$, which allow us to
improve the $\mathcal{Z}$-structure $\left(  \overline{X\times Y}%
,\partial_{\infty}\left(  X\times Y\right)  \right)  $ from Theorem
\ref{t:main} to an $\mathcal{EZ}$-structure. In other words, we are looking to
extend the conjugated action of $G$ on $X\times Y$ to $\partial_{\infty
}\left(  X\times Y\right)  $. Rather than striving for the most general
result, we prove a theorem that is sufficiently general for all applications
presented in this paper.

Recall that, for proper CAT(0) spaces $X$ and $Y$, $X\times Y$ (with the
$\ell_{2}$-metric) is CAT(0) with $\partial_{\infty}(X\times Y)\approx
\partial_{\infty}X\ast\partial_{\infty}Y$. One of the conditions we will
impose on the $G$-action on $X\times Y$ is that it splits as a product of $G$
actions. Neither of those action is expected to be geometric, but another
hypothesis will ensure that they extend over $\partial_{\infty}X$ and
$\partial_{\infty}Y$. Our action on $\partial_{\infty}X\ast\partial_{\infty}Y$
will be the join of those actions.

For the purposes of this section join lines of $\partial_{\infty}X\ast
\partial_{\infty}Y$ are parameterized by $\left[  0,\infty\right]  $, so as to
indicate slopes in $X\times Y$. The following lemma is based on standard
CAT(0) geometry. We leave its proof to the reader.

\begin{lemma}
\label{Lemma: generalizing joins}Let $\left(  X,d_{X}\right)  $ and $\left(
Y,d_{Y}\right)  $ be proper CAT(0) metric spaces; $\alpha:[0,\infty
)\rightarrow X$ and $\beta:[0,\infty)\rightarrow Y$ be proper topological
embeddings emanating from $x_{0}$ and $y_{0}$, respectively, and converging to
points $\overline{z}\in\partial_{\infty}X$ and $\overline{w}\in\partial
_{\infty}Y$. Let $Q=\left\{  \left(  \alpha\left(  t\right)  ,b\left(
t\right)  \right)  \mid t\in\lbrack0,\infty)\right\}  $. Then, the closure of
$Q$ in $\left(  X\times Y,d_{2}\right)  $ is $\overline{Q}=Q\cup
A_{\overline{z}\overline{w}}$ where $A_{\overline{z}\overline{w}}$ is the join
line in $\partial_{\infty}(X\times Y)$ connecting $\overline{z}$ to
$\overline{w}$. Furthermore, a proper topological ray (embedded or otherwise)
$\gamma=\left(  \gamma_{1},\gamma_{2}\right)  :[0,\infty)\rightarrow
Q\subseteq X\times Y$ converges to a point $m\in A_{\overline{z}\overline{w}}$
(of slope $m\in\lbrack0,\infty]$) if and only if
\[
\lim_{t\rightarrow\infty}\frac{d_{Y}\left(  \gamma_{2}\left(  t\right)
,y_{0}\right)  }{d_{X}\left(  \gamma_{1}\left(  t\right)  ,x_{0}\right)  }=m.
\]

\end{lemma}

Let $X$ and $Y$ be proper CAT(0) spaces and assume

\begin{description}
\item[i)] $h=(h_{1},h_{2}):X\times Y\rightarrow X\times Y$ is a
factor-preserving homeomorphism

\item[ii)] $h_{1}:X\rightarrow X$ is an isometry,

\item[iii)] $h_{2}:Y\rightarrow Y$ is a homeomorphism and a quasi-isometry
which extends to a homeomorphism on $\overline{Y}$, and

\item[iv)] $Y$ is radially compressible toward a basepoint $y_{0}$.
\end{description}

Let $\psi:\mathbb{R}^{+}\rightarrow\mathbb{R}^{+}$ is a proper function, and
let $h_{\psi}:Y\rightarrow Y$ be a corresponding radial compression function
based on a homeomorphism $\widehat{h}_{\psi}:\mathbb{R}^{+}\rightarrow
\mathbb{R}^{+}$ (see Remark \ref{Remark: radial compressions}). We say that
$\widehat{h}_{\psi}$ is\emph{ linearly controlled} if
\[
\lim_{t\rightarrow\infty}\frac{\widehat{h}_{\psi}g\widehat{h}_{\psi}^{-1}%
(t)}{t}=1
\]
for all linear maps $g:\rr^{+}\rightarrow\rr^{+}$.

\begin{remark}
By pre-composing $\widehat{h}_{\psi}$ with $\log(x+1)$, we can always assume
that $\widehat{h}_{\psi}$ is a linearly controlled compressing homeomorphism
for $\psi$.
\end{remark}

Let $\eta_{m}=\left(  \frac{1}{m}\eta_{1},\eta_{2}\right)  $ be a ray in
$X\times Y$, where $\eta_{1}$ and $\eta_{2}$ are geodesic rays emanating from
$x_{0}$ and $y_{0}$, and $\frac{1}{m}\eta_{1}(t)\equiv\eta_{1}(\frac{t}{m})$.
Viewing $\eta_{1}$ and $\eta_{2}$ as elements of $\partial_{\infty}X$ and
$\partial_{\infty}Y$, respectively, $\eta_{m}$ represents the point of the
join line $A_{\eta_{1}\eta_{2}}\subseteq\partial_{\infty}X\ast\partial
_{\infty}Y=\partial_{\infty}(X\times Y)$ at slope $m$ (a generic point of
$\partial_{\infty}X\ast\partial_{\infty}Y$). Next let $H_{\psi}%
=(\operatorname*{id}\nolimits_{X},h_{\psi}):X\times Y\rightarrow X\times Y$
(as in the proof of Theorem \ref{t:main}), and consider the topological ray
\[
\eta^{\prime}=H_{\psi}hH_{\psi}^{-1}\eta_{m}=\left(  h_{1}\left(  \frac{1}%
{m}\eta_{1}\right)  ,h_{\psi}h_{2}h_{\psi}^{-1}\eta_{2}\right)  .
\]
We wish to apply Lemma \ref{Lemma: generalizing joins} with $h_{1}\eta_{1}$
and $h_{\psi}h_{2}h_{\psi}^{-1}\eta_{2}$ playing the roles of $\alpha$ and
$\beta$. Clearly the radial homeomorphisms $h_{\psi}$ and $h_{\psi}^{-1}$
extend via the identity to $\partial_{\infty}Y$. Since it is an isometry,
$h_{1}$ extends to a homeomorphism of $\overline{X}$; and by hypothesis,
$h_{2}$ extends to a homeomorphism of $\overline{Y}$. Therefore $h_{1}\eta
_{1}$ and $h_{\psi}h_{2}h_{\psi}^{-1}\eta_{2}$ satisfy the hypothesis on
$\alpha$ and $\beta$, with $h_{1}\eta_{1}$ converging to $h_{1}\left(
\eta_{1}\right)  \in\partial_{\infty}X$ and $h_{\psi}h_{2}h_{\psi}^{-1}%
\eta_{2}$ converging to $h_{2}\left(  \eta_{2}\right)  \in\partial_{\infty}Y$.
Let $x_{0}^{\prime}:=h_{1}\eta_{1}\left(  0\right)  =h_{1}\left(
x_{0}\right)  $ and $y_{0}^{\prime}:=h_{\psi}h_{2}h_{\psi}^{-1}\eta_{2}\left(
0\right)  =h_{\psi}h_{2}\left(  y_{0}\right)  $.

The role of $\gamma=\left(  \gamma_{1},\gamma_{2}\right)  $ in our application
of Lemma \ref{Lemma: generalizing joins} is played by $\eta_{m}=\left(
\frac{1}{m}\eta_{1},\eta_{2}\right)  $. As such, consider
\[
\lim_{t\rightarrow\infty}\frac{d_{Y}\left(  h_{\psi}h_{2}h_{\psi}^{-1}\eta
_{2}\left(  t\right)  ,y_{0}^{\prime}\right)  }{d_{X}\left(  h_{1}(\frac{1}%
{m}\eta_{1})\left(  t\right)  ,x_{0}^{\prime}\right)  }=\lim_{t\rightarrow
\infty}\frac{\frac{d_{Y}\left(  h_{\psi}h_{2}h_{\psi}^{-1}\eta_{2}\left(
t\right)  ,y_{0}^{\prime}\right)  }{t}}{\frac{d_{X}\left(  h_{1}(\frac{1}%
{m}\eta_{1})\left(  t\right)  ,x_{0}^{\prime}\right)  }{t}}%
\]
It is easy to see that
\[
\lim_{t\rightarrow\infty}\frac{d_{X}\left(  h_{1}(\frac{1}{m}\eta_{1})\left(
t\right)  ,x_{0}^{\prime}\right)  }{t}=\frac{1}{m}%
\]
Since $h_{2}$ is a quasi-isometry, choose $K\geq1$ and $\varepsilon\geq0$ such
that
\[
\frac{1}{K}d\left(  y,y^{\prime}\right)  -\varepsilon\leq d\left(
h_{2}\left(  y\right)  ,h_{2}\left(  y^{\prime}\right)  \right)  \leq
Kd\left(  y,y^{\prime}\right)  +\varepsilon
\]
for all $y,y^{\prime}\in Y$. In particular,%

\[
\frac{1}{K}\cdot\widehat{h}_{\psi}^{-1}\left(  t\right)  -\varepsilon\leq
d\left(  h_{2}h_{\psi}^{-1}\eta_{2}\left(  t\right)  ,h_{2}\left(
y_{0}\right)  \right)  \leq K\cdot\widehat{h}_{\psi}^{-1}\left(  t\right)
+\varepsilon
\]

So, letting $C=d\left(  h_{2}\left(  y_{0}\right)  ,y_{0}\right)  $, we have
\[
\frac{1}{K}\cdot\widehat{h}_{\psi}^{-1}\left(  t\right)  -C\leq d\left(
h_{2}h_{\psi}^{-1}\eta_{2}\left(  t\right)  ,y_{0}\right)  \leq K\cdot
\widehat{h}_{\psi}^{-1}\left(  t\right)  +C
\]

Since $\widehat{h}_{\psi}$ can be taken to be monotone increasing, we have
\[
\widehat{h}_{\psi}\left(  \frac{1}{K}\cdot\widehat{h}_{\psi}^{-1}\left(
t\right)  -\left(  \varepsilon+C\right)  \right)  \leq\widehat{h}_{\psi
}\left(  d\left(  h_{2}h_{\psi}^{-1}\eta_{2}\left(  t\right)  ,y_{0}\right)
\right)  \leq\widehat{h}_{\psi}\left(  K\cdot\widehat{h}_{\psi}^{-1}\left(
t\right)  +\varepsilon+C\right)
\]
for $t$ sufficiently large.

Notice now that
\[
d\left(  h_{\psi}h_{2}h_{\psi}^{-1}\eta_{2}\left(  t\right)  ,y_{0}\right)
=h_{\psi}\left(  d\left(  h_{2}h_{\psi}^{-1}\eta_{2}\left(  t\right)
,y_{0}\right)  \right)
\]
so, for sufficiently large $t$, we have
\[
\widehat{h}_{\psi}\left(  \frac{1}{K}\cdot\widehat{h}_{\psi}^{-1}\left(
t\right)  -\left(  \varepsilon+C\right)  \right)  \leq d\left(  h_{\psi}%
h_{2}h_{\psi}^{-1}\eta_{2}\left(  t\right)  ,y_{0}\right)  \leq\widehat{h}%
_{\psi}\left(  K\cdot\widehat{h}_{\psi}^{-1}\left(  t\right)  +\varepsilon
+C\right)
\]

Now divide all three terms by $t$ and let $t\rightarrow\infty$. By the linear
control assumption on $\widehat{h}_{\psi}$, the corresponding left- and
right-hand limits are both $1$, hence the middle limit is $1$. Finally note
that
\[
d\left(  h_{\psi}h_{2}h_{\psi}^{-1}\eta_{2}\left(  t\right)  ,y_{0}\right)
-d\left(  y_{0},y_{0}^{\prime}\right)  \leq d\left(  h_{\psi}h_{2}h_{\psi
}^{-1}\eta_{2}\left(  t\right)  ,y_{0}^{\prime}\right)  \leq d\left(  h_{\psi
}h_{2}h_{\psi}^{-1}\eta_{2}\left(  t\right)  ,y_{0}\right)  +d\left(
y_{0},y_{0}^{\prime}\right)
\]
Divide all three terms by $t$ and let $t\rightarrow\infty$. Apply the above
work to again conclude that the left- and right-hand limits are $1$, so the
middle limit is $1$ as well. Putting these pieces together and applying Lemma
\ref{Lemma: generalizing joins}, the ray $\eta^{\prime}=H_{\psi}hH_{\psi}%
^{-1}\eta_{m}$ converges to the point of slope $m$ in $\partial_{\infty}%
X\ast\partial_{\infty}Y$ on the join line between $h_{1}\left(  \eta
_{1}\right)  \in\partial_{\infty}X$ and $h_{2}\left(  \eta_{2}\right)
\in\partial_{\infty}Y.$

Therefore, we have the following theorem.

\begin{Theorem}\label{t:ezstructure}
Let $X,Y$, $h$ and $h_{\psi}$ as above. Then the compressed homeomorphism $h_\psi h h_\psi^{-1}$ extends to the boundary $\partial_\infty X \ast \partial_\infty Y$. The induced homeomorphism on $\partial_\infty X \ast \partial_\infty Y$ is the join of the induced homeomorphisms on $\partial_\infty X$ and $\partial_\infty Y$.
\end{Theorem}

\section{Graphs of closed aspherical $n$%
-manifolds\label{Section: Graphs of closed aspherical n-manifolds}}

We begin to focus on our main class of examples, which can be thought of as
higher-dimensional analogues of generalized Baumslag-Solitar groups.

Suppose that $G$ is the fundamental group of a finite connected graph of
groups\emph{ }$\left(  \mathcal{G},\Gamma\right)  $, with the property that
each vertex group $G_{v}$ is the fundamental group of a closed aspherical
manifold $M_{v}$ and, for each edge $e$, the monomorphisms $G_{e}%
\overset{\phi_{e}^{-}}{\longrightarrow}G_{i(e)}$ and $G_{e}\overset{\phi
_{e}^{+}}{\longrightarrow}G_{t(e)}$ are of finite index. The primary goal in
this section is to realize $\left(  \mathcal{G},\Gamma\right)  $ with a graph
of covering spaces (as defined and developed in the appendix). The assumptions
on $\left(  \mathcal{G},\Gamma\right)  $ ensure that each $G_{e}$ can be
realized as the fundamental group of both a finite-sheeted cover $M_{e}^{-}$
of $M_{i\left(  e\right)  }$ and a finite-sheeted cover $M_{e}^{+}$ of
$M_{t\left(  e\right)  }$. These covers are homotopy equivalent, but a priori,
not homeomorphic. Being closed and aspherical, all vertex manifolds and their
covers are necessarily of the same dimension. To proceed, we need a single
edge space, for each edge $e$, which covers \emph{both} of its vertex spaces
(possibly the same space, in cases where $e$ is a loop in $\Gamma$). We will
describe two useful approaches. Each approach requires an additional hypothesis that
is conjecturally satisfied in all cases. The first is conceptually simpler; it
chooses one of the above-mentioned covers as the edge space and leaves the
chosen vertex manifolds in place. The second approach is more drastic, but has
some important benefits.\medskip

\noindent\textbf{Approach I. }\emph{Assume the Borel Conjecture holds for all
edge groups.}\medskip

At each vertex $v$ of $\Gamma$, place the aspherical manifold $M_{v}$ chosen
above, then realize $\phi_{e}^{-}$ and $\phi_{e}^{+}$ by finite-sheeted
covering projections $q_{e}^{-}:M_{e}^{-}\rightarrow M_{i(e)}$ and $q_{e}%
^{+}:M_{e}^{+}\rightarrow M_{t(e)}$. By hypothesis, there is a homeomorphism
$f:M_{e}^{-}\rightarrow M_{e}^{+}$. Let $M_{e}=M_{e}^{-}$ and define
$p_{e}^{-}=q_{e}^{-}$and $p_{e}^{+}=q_{e}^{+}\circ f$ to be the edge
maps.\medskip

\noindent\textbf{Approach II. }\emph{Assume that the Whitehead group of each
edge group }$G_{e}$\emph{ is trivial}.\medskip

Using the same notation as above and the new hypothesis, $M_{e}^{-}$ and
$M_{e}^{+}$ are simple homotopy equivalent, so by the results discussed in
Section \ref{Section: The Hilbert cube and Hilbert cube manifolds}, there is a
homeomorphism $f:M_{e}^{-}\times I^{\omega}\rightarrow M_{e}^{+}\times
I^{\omega}$. Replace each $M_{v}$ with $M_{v}\times I^{\omega}$ and for each
edge $e$, let $M_{e}^{-}\times I^{\omega}$ be the edge space (denoted simply
by $M_{e}\times I^{\omega}$ from now on). Then insert the covering maps
$p_{e}^{-}=q_{e}^{-}\times\operatorname*{id}_{I^{\omega}}$ and $p_{e}%
^{+}=\left(  q_{e}^{+}\times\operatorname*{id}_{I^{\omega}}\right)  \circ f$
to complete the realization of $\left(  \mathcal{G},\Gamma\right)  $ as a
graph of covering spaces.\medskip

Given the hypothesis and setup in Approach I, we can form the \emph{total
space}
\[
X=(\bigcup_{v}M_{v})\cup(\bigcup_{e}M_{e}\times\left[  0,1\right]  )
\]
where $M_{e}\times\left\{  0\right\}  $ and $M_{e}\times\left\{  1\right\}  $
are glued to $M_{i(e)}$ and $M_{t(e)}$ using covering maps $p_{e}^{-}%
:M_{e}\rightarrow$ $M_{i(e)}$ and $p_{e}^{+}:M_{e}\rightarrow$ $M_{t(e)}$
defined above. From there we pass to the universal cover $\widetilde{X}$ to
get the desired $G$-space. Any geodesic metric on $X$ lifts to a $G$-invariant
metric on the $\widetilde{X}$. This cover together with the the $G$-invariant
metric satisfy the following properties. (See \cite{fm} and the appendix to
this paper for details.)

\begin{itemize}
\item There is a distance non-increasing projection map $p_{T}:\widetilde{X}%
\rightarrow T$, where $T$ is the Bass-Serre tree for $\left(  \mathcal{G}%
,\Gamma\right)  $.

\item There is a homeomorphism $H:\widetilde{X}\rightarrow T\times
\widetilde{M}_{v}$, where $\widetilde{M}_{v}$ is the universal cover of an
arbitrary vertex space. Furthermore, $p_{T}^{-1}(t)$ maps to $t\times
\widetilde{M}_{v}$ under $H$ and, for all $m\in\widetilde{M}_{v}$, the map
$T\rightarrow T\times m\rightarrow\widetilde{X}$ is a locally isometric embedding.

\item There exists $C\geq1$ such that for all edges $e$ of $T$ and $v\in e$,
the retraction $r:e\rightarrow v$ induces a projection
\[
p_{T}^{-1}(e)\overset{H}{\rightarrow}e\times\widetilde{M}_{v}\rightarrow
v\times\widetilde{M}_{v}\overset{H^{-1}}{\rightarrow}p_{t}^{-1}(v)
\]
which is $C$-Lipschitz.
\end{itemize}

The reader can compare these statements with the well-known picture of the
Cayley complex of the Baumslag-Solitar group $BS(m,n)$ (homeomorphic to
$T\times\rr$, where $T$ is the Bass-Serre tree of the splitting and
$\rr=\widetilde{S^{1}}$).\medskip

In the case of Approach II, everything works the same as above, except that
the vertex and edge spaces are now the aspherical Hilbert cube manifolds
$M_{v}\times I^{\omega}$ and $M_{e}\times I^{\omega}$ and their universal
covers are $\widetilde{M}_{v}\times I^{\omega}$ and $\widetilde{M}_{e}\times
I^{\omega}$. The spaces $X$ and $\widetilde{X}$ are now Hilbert cube manifolds.

The above may be summarized as follows.

\begin{Theorem}
\label{Theorem: main realization theorem}Let $\left(  \mathcal{G}%
,\Gamma\right)  $ be a finite connected graph of groups\emph{ }$\left(
\mathcal{G},\Gamma\right)  $, with the property that each vertex group $G_{v}$
is the fundamental group of a closed aspherical manifold $M_{v}$ and for each
edge $e$, the monomorphisms $G_{e}\overset{\phi_{e}^{-}}{\longrightarrow
}G_{i(e)}$ and $G_{e}\overset{\phi_{e}^{+}}{\longrightarrow}G_{t(e)}$ are of
finite index. Then

\begin{enumerate}
\item[I.] if the Borel Conjecture holds for each edge group $G_{e}$, then
$\left(  \mathcal{G},\Gamma\right)  $ can be realized by a graph of covering
spaces with vertex spaces $M_{v}$,

\item[II.] if the Whitehead group $\operatorname*{Wh}\left(  G_{e}\right)  $
vanishes for each $G_{e}$, then $\left(  \mathcal{G},\Gamma\right)  $ can be
realized by a graph of covering spaces where the vertex spaces are the Hilbert
cube manifolds $M_{v}\times I^{\omega}$.
\end{enumerate}
\end{Theorem}

\begin{remark}
For the purposes of this paper, Approaches I and II lead to nearly identical
places. That is largely due to our reliance on a compressibility hypothesis,
which we can verify only for nonpositively curved Riemannian manifolds.
Farrell and Jones \cite{fj} have shown that, with one significant exception,
both the Borel Conjecture and the triviality of the Whitehead group hold for
(fundamental groups of) closed nonpositively curved Riemannian manifolds. The
exception occurs when $n=4$, where their surgery-theoretic proof of the Borel Conjecture does
not apply. As such, Approach II is essential for obtaining the $n=4$ case of
our main results. Approach II also holds promise for proving more general
theorems, but that is likely to require a different method---one that does not
involve compressibility. Further discussion of that idea is included in
Section \ref{Section: Some open questions}.
\end{remark}

\section{Graphs of nonpositively curved Riemannian $n$%
-manifolds\label{Section: Graphs of nonpositively curved n-manifolds}}

We are now ready to present our main results, in which we apply Theorems
\ref{Theorem: main realization theorem}, \ref{t:main} and \ref{t:ezstructure}
to provide new classes of groups that admit $(\mathcal{E})\cz$-structures.

\begin{Theorem}
\label{Theorem: E/Z-structures for graphs of manifolds}Suppose $G$ is the
fundamental group of a finite graph of groups, where each vertex group is the
fundamental group of a closed, nonpositively curved Riemannian manifold, and
each edge group is finite index in corresponding vertex groups. Then $G$
admits a $\mathcal{Z}$-structure. If the lifts to universal covers of all
covering maps $p_{e}^{-}:M_{e}\rightarrow$ $M_{i(e)}$ and
$p_{e}^{+}:M_{e}\rightarrow$ $M_{t(e)}$ (discussed above) extend over the
visual boundaries, then $G$ admits an $\mathcal{EZ}$-structure.
\end{Theorem}

\begin{proof}
For the sake of simplicity, begin my assuming that $n\neq4$. Then, by
\cite{fj}, the Borel Conjecture holds for each edge group, so we may use the
graph of covering spaces described in Approach I above.

Our initial task is to check the conditions found in Theorem \ref{t:main}.
Using the homeomorphism $\widetilde{X}\rightarrow T\times\widetilde{M}_{v}$
noted above, $G$ acts properly and cocompactly on the product $T\times
\widetilde{M}_{v}$. (For a more detailed discussion of this action, see the
appendix.) By Theorem \ref{manifold compressible} and Proposition \ref{p:qi},
$\widetilde{M}_{v}$ is compressible for any $\pi_{1}(M_{v})$-equivariant
metric. Choose any nonpositively curved Riemannian metric. We fix a basepoint $t_{0}\in T$; make
$\widetilde{M}_{v}$ isometric to $p_{T}^{-1}(t_{0})$; put the usual metric on
$T$; and give $T\times\widetilde{M}_{v}$ the product metric. Choose a compact
set $K$ in $\widetilde{X}$ so that $GK=\widetilde{X}$. Note that
$\diam(p_{T}(gK))\leq\diam(K)$, so condition (1) of Theorem \ref{t:main} is
satisfied. Let $p_{\widetilde{M}_{v}}$ be the projection $\widetilde{X}%
\rightarrow p_{T}^{-1}(t_{0})\cong t_{0}\times\widetilde{M}_{v}$ and let
$D=\diam(p_{\widetilde{M}_{v}}K)$. Now, suppose $p_{T}(gK)\subset B_{T}%
(t_{0},R)$. The projection
\[
p_{T}^{-1}(B_{T}(t_{0},R))\rightarrow B_{T}(t_{0},R)\times\widetilde{M}%
_{v}\rightarrow t_{0}\times\widetilde{M}_{v}\rightarrow p_{T}^{-1}(t_{0})
\]
is $C^{R}$-Lipschitz, so $p_{\widetilde{M}_{v}}(gK)$ has diameter $<DC^{R}$.
Thus, condition (2) of Theorem \ref{t:main} is satisfied. It follows that $G$
admits a $\mathcal{Z}$-structure.

Now assume that all lifts $\widetilde{M}_{i\left(  e\right)  }%
\overset{\widetilde{p_{e}^{-}}}{\longleftarrow}\widetilde{M}_{e}%
\overset{\widetilde{p_{e}^{+}}}{\longrightarrow}\widetilde{M}_{t\left(
e\right)  }$ of our finite-sheeted coverings extend over their corresponding
visual boundaries. To obtain an $\mathcal{EZ}$-structure, it suffices to
verify the conditions in Theorem \ref{t:ezstructure} for all elements of $G$,
viewed as self-homeomorphisms of $T\times\widetilde{M}_{v}$. By the proof of
Theorem \ref{manifold compressible}, we know that $\widetilde{M}_{v}$ is
radially contractible, so it suffices to check i)-iii). Items i) and ii) are
discussed in detail in Section \ref{Subsection: Graphs of covering spaces} of
the appendix, with the action on the first factor being the standard
Bass-Serre action. As is discussed in Remark
\ref{Remark: Action by quasi-isometric homeomorphisms}, each element of $G$
acting on $\widetilde{M}_{v}$ is a finite composition of lift homeomorphisms,
inverses of those homeomorphisms, and isometries of vertex spaces. Each of
those is a quasi-isometric homeomorphism, and by hypothesis, each extends over
the corresponding boundaries. Therefore, condition iii) holds as well.

Next, in order to cover the $n=4$ case (and to offer an alternative proof in
all other dimensions), let us switch to the setup described in Approach II.
Again, \cite{fj} confirms the necessary hypothesis. In order to apply Theorem
\ref{t:main}, we need a $\pi_{1}(M_{v})$-equivariant CAT(0) metric on
$\widetilde{M}_{v}\times I^{\omega}$. This can be accomplished by using the
$\ell_{2}$-metric described in Section
\ref{Section: The Hilbert cube and Hilbert cube manifolds}. We also need to
know that $\widetilde{M}_{v}\times I^{\omega}$ is compressible---a fact that
was noted in Item \ref{Item: compressible but not radially} of Remark
\ref{Remark: radial compressions}. Everything else in the above proof now goes
through without changes.
\end{proof}

Note that if each manifold is negatively curved, the lift of any finite
covering map is a quasi-isometry between Gromov hyperbolic spaces, and hence
extends to the visual boundaries.

\begin{corollary}
\label{Corollary: main corollary 1}Graphs of nonpositively curved closed
Riemannian $n$-manifolds admit $\mathcal{Z}$-structures. Graphs of negatively
curved Riemannian $n$-manifolds admit $\mathcal{EZ}$-structures.
\end{corollary}

For generic nonpositively curved Riemannian manifolds, we cannot be sure that
the lifts of all the covering maps constructed in Section
\ref{Section: Graphs of closed aspherical n-manifolds} extend over visual
boundaries. The problem is the possibly non-geometric nature of the lifts of
the homeomorphisms $f:M_{e}^{-}\rightarrow M_{e}^{+}$ (or $f:M_{e}^{-}\times
I^{\omega}\rightarrow M_{e}^{+}\times I^{\omega}$) used in defining $p_{e}%
^{+}:M_{e}\rightarrow M_{t(e)}$ (or $p_{e}^{+}:M_{e}\times I^{\omega
}\rightarrow M_{t(e)}\times I^{\omega}$). By applying the Bieberbach Theorems
 \cite{Szc12}, we can avoid this problem in the extreme (but
important) special case where all vertex manifolds are flat.

\begin{corollary}
\label{Corollary: main corollary 2}Graphs of closed flat $n$-manifolds admit
$\ce\mathcal{Z}-$structures.
\end{corollary}

\begin{proof}
The third Bieberbach Theorem, as described in \cite[Ch.2]{Szc12}, assures
that, for all $n$ and any pair of closed flat $n$-manifolds with isomorphic
fundamental groups, there is a corresponding affine homeomorphism; in other
words, a homeomorphism that lifts to an affine homeomorphism $\rr^{n}%
\rightarrow\rr^{n}$. Among other things, this will allow us to use Approach I
from Section \ref{Section: Graphs of closed aspherical n-manifolds}, even when
$n=4$.

Choose a flat metric on each vertex manifold $M_{v}$, then use the covering
maps $q_{e}^{-}:M_{e}^{-}\rightarrow M_{i(e)}$ and $q_{e}^{+}:M_{e}%
^{+}\rightarrow M_{t(e)}$ to lift those metrics to the finite-sheeted covers
$M_{e}^{-}$ and $M_{e}^{+}$. As such, the lifts to universal covers
$\widetilde{q}_{e}^{-}$ and $\widetilde{q}_{e}^{+}$ become isometries of $%
%TCIMACRO{\U{211d} }%
%BeginExpansion
\mathbb{R}
%EndExpansion
^{n}$. The Bieberbach Theorem then allows us to choose a homeomorphism
$f:M_{e}^{-}\rightarrow M_{e}^{+}$ that lifts to an affine isomorphism of $%
%TCIMACRO{\U{211d} }%
%BeginExpansion
\mathbb{R}
%EndExpansion
^{n}$. Since isometries and affine isomorphisms of $%
%TCIMACRO{\U{211d} }%
%BeginExpansion
\mathbb{R}
%EndExpansion
^{n}$ all extend over visual boundaries, our Corollary follows.
\end{proof}

Using these results, we also obtain a strengthening of the result from
\cite{gmt}:

\begin{corollary}
\label{Corollary: main corollary 3}Generalized Baumslag-Solitar groups admit
$\mathcal{EZ}$-structures.
\end{corollary}

\begin{remark}
A few comments are in order as we close this section.

\begin{enumerate}
\item The groups addressed in Corollaries \ref{Corollary: main corollary 1}
and \ref{Corollary: main corollary 2} have been previously studied by a number
of people, see \cite{fm} and \cite{msw}. Among other things, these fundamental
groups are quasi-isometrically rigid in the sense that any group
quasi-isometric to such a group is itself the fundamental group of a finite
graph of groups with vertex/edge groups quasi-isometric to the original
vertex/edge groups.

\item The proof of Theorem
\ref{Theorem: E/Z-structures for graphs of manifolds} is valid for graphs of
(non-Riemannian) nonpositively curved manifolds, provided they are
compressible (to get a $\mathcal{Z}$-structure) or radially compressible (to
get an $\mathcal{EZ}$-structure). At this time, we do not know any examples of
that type.

\item A primary motivation for studying $\mathcal{EZ}$-structures is that a
torsion-free group with a $\mathcal{EZ}$-structure satisfies the Novikov
Conjecture. See \cite{fl} and also \cite{CaPe95}. All of the examples covered
by Corollaries \ref{Corollary: main corollary 1}%
-\ref{Corollary: main corollary 3} were previously known to satisfy the
Novikov Conjecture for other reasons. For example, hyperbolic and free abelian
groups have finite asymptotic dimension, so by work of Bell and Dranishnikov
\cite{bd}, so do graphs of groups with these as vertex and edge groups . It is
an open question whether fundamental groups of all nonpositively curved
manifolds (Riemannian or otherwise) have finite asymptotic dimension. As such,
it is possible that Theorem
\ref{Theorem: E/Z-structures for graphs of manifolds} contains new examples of
groups which satisfy the Novikov Conjecture. 
\end{enumerate}
\end{remark}

\section{Noncompressible spaces\label{Section: Noncompressible spaces}}

In this section, we highlight the delicate nature of compressibility by
looking at some noncompressible CAT(0) spaces which occur as universal
coverings of compact aspherical CW-complexes and manifolds. We begin with the
simplest such example.

\begin{example}
Let $T_{4}$ be the tree with valence 4 at each vertex and standard path length
metric, i.e., the universal cover of a wedge of two circles. If we view
$T_{4}$ as a 1-manifold with singularites at the vertices, it is clear that
large balls have more singular points than small balls. This is an obstruction
to the existence of compressing homeomorphisms.
\end{example}

\begin{example}
Now consider $T_{4}\times I^{\omega}$ with the $\ell_{2}$-metric. By
\cite{Wes71}, this is a Hilbert cube manifold, therefore a homogeneous space.
This nullifies the above argument, but compressibility still fails since large
balls in $\mathbb{T}_{4}\times I^{\omega}$ have more complementary components
than small balls.
\end{example}

Next we examine an example of more direct relevance to this paper. In
particular, we identify a family of closed, nonpositively curved (locally
CAT(0)) finite-dimensional manifolds whose universal covers are not compressible.

By a \emph{standard Davis example} we are referring to the special case of the
construction in \cite{Dav83}. This begins with a compact contractible
$q$-manifold $Q^{q}$ with a mirror structure $\left\{  Q_{v}\right\}  _{v\in
V}$ consisting of tame $\left(  q-1\right)  $-cells in $\partial Q^{q}$, and a
Coxeter system $\left(  \Gamma,V\right)  $, consisting of a Coxeter group
$\Gamma$ and a preferred generating set $V$ in one-to-one correspondence with
the mirrors. We assume that $\partial Q_{q}=\cup Q_{v}$ and the mirror
structure is \textquotedblleft$\Gamma$-finite\textquotedblright. For any
compact contractible $q$-manifold $Q^{q}$, such an arrangement exists: begin
with a flag triangulation $K$ of $\partial Q^{q}$ and let the mirrors be the
top-dimensional cells of the corresponding dual cell-structure on $\partial
Q^{q}$; they are indexed by the vertex set $V=K^{0}$. A corresponding
(right-angled) Coxeter system $\left(  \Gamma,V\right)  $ is obtained by
declaring $v_{i}^{2}=1$ for all $v_{i}\in V$ and $\left(  v_{i}v_{j}\right)
^{2}=1$ when $v_{i}$ and $v_{j}$ bound an edge in $K$.

Roughly speaking, $\Gamma$ provides instructions for gluing together members
of the discrete collection $\Gamma\times Q^{q}$ of copies of $Q^{q}$, to
obtain a contractible open manifold $X^{q}$ that admits a proper cocompact
$\Gamma$-action. Within $X^{q}$, the individual copies of $Q^{q}$ are referred
to as \emph{chambers}, with the chamber corresponding to $\left\{  g\right\}
\times Q^{q}$ denoted by $gQ^{q}$, and the identity chamber $\left\{
e\right\}  \times Q^{q}$ denoted as $Q^{q}$. By passing to a torsion-free
finite index subgroup $\Gamma^{\prime}\leq\Gamma$, one obtains a covering
projection $X^{q}\rightarrow\Gamma^{\prime}\backslash X^{q}$ with quotient a
closed aspherical manifold.

From the perspective of this paper, the key facts about $X^{q}$ are contained
in Lemma 8.2 and Remark 10.6 of \cite{Dav83}. It is observed that, if the
elements of $\Gamma$ are ordered $1=g_{1},g_{2},g_{3},\cdots$ so that
$\operatorname*{length}\left(  g_{j+1}\right)  \geq\operatorname*{length}%
\left(  g_{j}\right)  $ and $T_{i}=\cup_{j=1}^{i}g_{j}Q^{q}$, then for all
$i$, $T_{i}$ is a connected $q$-manifold with boundary, and $T_{i}\cap
g_{i+1}Q^{q}$ is a tame $\left(  q-1\right)  $-cell in the boundary of each
(made up of a finite union of panels). Since $Q^{q}$ is simply connected, it
is orientable; so assume now that $Q^{q}$ is an oriented manifold, and give
each chamber $g_{i}Q^{q}$ that same orientation when $\operatorname*{length}%
\left(  g_{i}\right)  $ is even and the opposite orientation when
$\operatorname*{length}\left(  g_{i}\right)  $ is odd. A quick look at the
gluing instructions in \cite{Dav83} for assembling the chambers into $X^{q}$
confirms that the orientations on the chambers fit together to provide
appropriate orientations on the $T_{i}$. All of this implies that $T_{i}$ is a
boundary connected sum of $i$ copies of $\pm Q$, hence $\partial T_{i}$ is a
connected sum of $i$ copies of $\pm\partial Q$. This is most interesting when
$q\geq4$ and $\pi_{1}\left(  \partial Q^{q}\right)  =G\neq1$, in which case,
$\pi_{1}\left(  \partial T_{i}\right)  $ is the free product $\ast_{k=1}^{i}%
G$. A key observation of Davis is that, for the corresponding neighborhood of
infinity $N_{i}=X^{q}-\operatorname*{int}T_{i}$, we have $\partial
T_{i}\hookrightarrow N_{i}$ is a homotopy equivalence, therefore $\pi
_{1}\left(  N_{i}\right)  = \ast_{k=1}^{i}G$. Davis used this fact to show
that $X^{q}$ is not simply connected at infinity (hence, not homeomorphic to
$\mathbb{R}^{q}$). We will use it for a similar, but different, reason.

Place a geodesic metric $d^{\prime}$ on $Q^{q}$; let $R=\operatorname*{diam}%
Q^{q}$; and give $X^{q}$ the corresponding path length metric $d$. As such,
$\left(  X^{q},d\right)  $ is a proper geodesic metric space, and the action
of $\Gamma$ on $X^{q}$ is geometric. Give $\Gamma$ the word length metric
$\rho$ corresponding to the generating set $V$. Choose $x_{0}\in
\operatorname*{int}Q^{q}$ and let $f:\left(  \Gamma,\rho\right)
\rightarrow\left(  X^{q},d\right)  $ be defined by $f\left(  g\right)
=gx_{0}$. Then $f$ has $R$-dense image, and by \v{S}varc-Milnor, is a $\left(
K,\varepsilon\right)  $-quasi-isometry for some $K\geq1$ and $\varepsilon
\geq0$. Note that, by our choice of $x_{0}$, $f$ is injective.

Let $\beta:\mathbb{N}\rightarrow\mathbb{N}$ be the growth function for
$\left(  \Gamma,\rho\right)  $, i.e., $\beta\left(  n\right)  =\left\vert
B_{\rho}[e,n]\right\vert $, and for $r>0$ and $A\subseteq X^{q}$, let
$N_{d}[A;r]$ denote the closed $r$-neighborhood of $A$ in $X^{q}$. Let $S$ be
the smallest integer such that $B_{d}\left[  x_{0};R\right]  \subseteq
T_{\beta\left(  S\right)  }$.

\begin{lemma}
\label{Lemma: Noncompressibility}For each $n\in\mathbb{N}$,

\begin{enumerate}
\item \label{Lemma: Noncompressibility item 1}$B_{d}\left[  x_{0},\frac{n}%
{K}-\varepsilon-R\right]  \subseteq N_{d}[f\left(  B_{\rho}[e,n]\right)  ,R]$

\item \label{Lemma: Noncompressibility item 2}$N_{d}[f\left(  B_{\rho
}[e,n]\right)  ,R]\subseteq T_{\beta\left(  n+S\right)  }$, and

\item \label{Lemma: Noncompressibility item 3}$T_{\beta\left(  n\right)
}\subseteq B_{d}\left[  x_{0},Kn+\varepsilon+R\right]  $.
\end{enumerate}
\end{lemma}

\begin{proof}
To prove \ref{Lemma: Noncompressibility item 1}), suppose $y\notin
N_{d}[f\left(  B_{\rho}[e,n]\right)  ,R]$. Then there exists $g\in\Gamma$,
such that $\rho\left(  e,g\right)  >n$ and $d\left(  y,f\left(  g\right)
\right)  \leq R$. Therefore
\[
\frac{1}{K}\rho\left(  e,g\right)  -\varepsilon\leq d\left(  x_{0},f\left(
g\right)  \right)  \leq d\left(  x_{0},y\right)  +R
\]
So
\[
\frac{n}{K}-\varepsilon-R<d\left(  x_{0},y\right)
\]

For item \ref{Lemma: Noncompressibility item 2}), let $gx_{0}\in f\left(
B_{\rho}[e,n]\right)  $. Then $\operatorname*{length}\left(  g\right)  \leq n$
and
\[
B_{d}\left[  gx_{0}:R\right]  \subseteq gT_{\beta\left(  S\right)  }=
{\displaystyle\bigcup\limits_{h\in B_{\rho}[e;S]}}ghQ^{q}%
\]
By definition, $\operatorname*{length}\left(  g\right)  \leq n$ and
$\operatorname*{length}\left(  h\right)  \leq S$, so $gh$ in the above
equality has length $\leq n+S$. Therefore the right-hand set is contained in
$T_{\beta\left(  n+S\right)  }$.

For the final item, suppose $gQ^{q}$ is a summand in $T_{\beta\left(
n\right)  }$. Then $\operatorname*{length}\left(  g\right)  \leq n$, so
$gx_{0}\in f\left(  \left[  B_{\rho}\left[  e,n\right]  \right]  \right)
\subseteq B_{d}\left[  x_{0},Kn+\varepsilon\right]  $. By the triangle
inequality, $gQ^{q}\subseteq B_{d}\left[  x_{0},Kn+\varepsilon+R\right]  $.
\end{proof}

\begin{Theorem}
\label{Theorem: Noncompressibility of Davis manifolds}Let $X^{q}$ be a Davis
manifold with chamber a compact contractible $q$-manifold $Q^{q}$ with
non-simply connected boundary. Let $\Gamma$ be the corresponding Coxeter group
and $d$ a metric on $X^{q}$ such that $\Gamma$ acts geometrically on $X^{q}$.
Then $X^{q}$ is noncompressible under the metric $d$.
\end{Theorem}

Note that Proposition \ref{p:qi} implies that $X^{q}$ is noncompressible for
any quasi-isometric metric. The following corollary follows immediately from
\cite{ADG97}.

\begin{corollary}
\label{Corollary: Noncompressible CAT(0) manifolds}For all $q\geq5$, there
exist closed locally CAT(0) $q$-manifolds with noncompressible universal covers.
\end{corollary}

\begin{proof}
[Proof of Theorem \ref{Theorem: Noncompressibility of Davis manifolds}]We will
use the metric $d$ and the constants $K$, $\varepsilon$, $R$, and $S$ defined above.

By Lemma \ref{Lemma: Noncompressibility},
\[
B_{d}\left[  x_{0},\frac{n}{K}-\varepsilon-R\right]  \subseteq T_{\beta\left(
n+S\right)  }\subsetneq T_{\beta\left(  n+S+1\right)  }\subseteq B_{d}\left[
x_{0},K\left(  n+S+1\right)  +\varepsilon+R\right]
\]
for all $n\in\mathbb{N}$. Suppose that $X^{q}$ is compressible, and let
$\psi:\mathbb{R}^{+}\rightarrow\mathbb{R}^{+}$ be the identity function. Then
there exists a homeomorphism $h_{\psi}:X^{q}\rightarrow X^{q}$ and a sublinear
function $\phi:\mathbb{R}^{+}\rightarrow\mathbb{R}^{+}$ such that
\[
\operatorname*{diam}\left(  h_{\psi}\left(  C\right)  \right)  \leq\phi\left(
\operatorname*{diam}C\right)
\]
for all bounded $C\subseteq X^{q}$. By composing with an isometry from
$\Gamma$, we may assume that $d\left(  x_{0},h_{\psi}\left(  x_{0}\right)
\right)  \leq R$ and by the above arrangement,
\[
\operatorname*{diam}h_{\psi}\left(  B_{d}\left[  x_{0},K\left(  n+S+1\right)
+\varepsilon+R\right]  \right)  \leq\phi\left(  2\left(  K\left(
n+S+1\right)  +\varepsilon+R\right)  \right)
\]
for all $n\in\mathbb{N}$. Since
\[
\lim_{n\rightarrow\infty}\frac{\phi\left(  n\right)  }{n}=0
\]
then
\[
\lim_{n\rightarrow\infty}\frac{\phi\left(  2\left(  K\left(  n+S+1\right)
+\varepsilon+R\right)  \right)  }{2\left(  \frac{n}{K}-\varepsilon-R\right)
}=0
\]
By choosing $n$ so large that $\operatorname*{diam}h_{\psi}\left(
B_{d}\left[  x_{0},K\left(  n+S+1\right)  +\varepsilon+R\right]  \right)
<\frac{1}{2}\cdot\left(  \frac{n}{K}-\varepsilon-R\right)  $ and $R<\frac
{1}{2}\cdot\left(  \frac{n}{K}-\varepsilon-R\right)  $, we obtain
\[
h_{\psi}\left(  B_{d}\left[  x_{0},K\left(  n+S+1\right)  +\varepsilon
+R\right]  \right)  \subseteq B_{d}\left(  x_{0},\frac{n}{K}-\varepsilon
-R\right)
\]
As a result, $h_{\psi}\left(  T_{\beta\left(  n+S+1\right)  }\right)
\subseteq\operatorname*{int}T_{\beta\left(  n+S\right)  }$.

Let $W=T_{\beta\left(  n+S\right)  }-\operatorname*{int}\left(  h_{\psi
}\left(  T_{\beta\left(  n+S+1\right)  }\right)  \right)  $ and consider the
cobordism $\left(  W,\partial T_{\beta\left(  n+S\right)  },h_{\psi}\left(
\partial T_{\beta\left(  n+S+1\right)  }\right)  \right)  $. By Lemma
\ref{Lemma: 1-sided h-cobordism} below, $W$ deformation retracts onto
$h_{\psi}\left(  \partial T_{\beta\left(  n+S+1\right)  }\right)  $ and the
restriction of this deformation is a degree $\pm1$ map $d:\partial
T_{\beta\left(  n+S\right)  }\rightarrow h_{\psi}\left(  \partial
T_{\beta\left(  n+S+1\right)  }\right)  $. It is a standard fact that degree
$\pm1$ maps induce $\pi_{1}$-surjections, so we have a surjection $d_{\ast
}:\ast_{k=1}^{\beta\left(  n+S\right)  }G\rightarrow\ast_{k=1}^{\beta\left(
n+S+1\right)  }G$. But then the rank domain is at least as large as the rank
of the range, violating Grushko's Theorem.
\end{proof}

We conclude ths section with the technical lemma used above.

\begin{lemma}
\label{Lemma: 1-sided h-cobordism}Let $M^{n}$ be an orientable open
$n$-manifold containing closed neighborhoods of infinity $N$ and $N^{\prime}$,
each a codimension $0$ submanifold with tame (bicollared) boundary. Suppose
also that $N^{\prime}\subseteq\operatorname*{int}N$ and both $\partial
N\hookrightarrow N$ and $\partial N^{\prime}\hookrightarrow N^{\prime}$ are
homotopy equivalences. Let $W=N-\operatorname*{int}N^{\prime}$. Then

\begin{enumerate}
\item $W$ is a compact $n$-manifold with $\partial W=\partial N\sqcup\partial
N^{\prime}$,

\item $W$ deformation retracts onto $\partial N$, and

\item the resulting retraction $r:W\rightarrow\partial N$ restricts to a
degree $\pm1$ map $\partial N^{\prime}\rightarrow\partial N$.
\end{enumerate}
\end{lemma}

\begin{proof}
Assertion 1) is immediate. For assertion 2), let $H_{t}$ and $J_{t}$ be
deformation retractions of $N$ onto $\partial N$ and $N^{\prime}$ onto
$\partial N^{\prime}$, respectively. Then $H_{1}:N\rightarrow\partial N$ and
$J_{1}:N^{\prime}\rightarrow\partial N^{\prime}$ are retractions, and
$J_{1}\circ H_{t}$ is a deformation retraction of $W$ onto $\partial N$.

For assertion 3), note that since $\partial N$ is a connected orientable
$\left(  n-1\right)  $-manifold, $H_{n-1}\left(  \partial N\right)
\cong\mathbb{Z}$; and since $r_{\ast}:H_{n-1}\left(  W\right)  \rightarrow
H_{n-1}\left(  \partial N\right)  $ is an isomorphism, $H_{n-1}\left(
W\right)  \cong\mathbb{Z}$. By duality $H_{\ast}\left(  W,\partial N^{\prime
}\right)  =0$, so $\operatorname*{incl}_{\ast}:H_{n-1}\left(  \partial
N^{\prime}\right)  \rightarrow H_{n-1}\left(  W\right)  $ is also an
isomorphism. It follows that $\left(  \left.  r\right\vert _{\partial
N^{\prime}}\right)  _{\ast}=r_{\ast}\circ\operatorname*{incl}_{\ast}$ is an
isomorphism, so $\left\vert \deg\left(  \left.  r\right\vert _{\partial
N^{\prime}}\right)  \right\vert =1$.
\end{proof}

Unfortunately, crossing a Davis manifold with the Hilbert cube does not
improve its compressibility properties.

\begin{Theorem}
Let $\left(  X^{q},d\right)  $ be a Davis manifold of the type described in
Theorem \ref{Theorem: Noncompressibility of Davis manifolds}. Then
$X^{q}\times I^{\omega}$ is noncompressible under the corresponding $\ell_{2}%
$-metric or any metric quasi-isometric to it.
\end{Theorem}

\begin{proof}
If we assume compressibility, the same sort of argument used above leads to a
cobordism of Hilbert cube manifolds $\left(  W\times I^{\omega},\partial
T_{\beta\left(  n+S\right)  }\times I^{\omega},h_{\psi}(\partial
T_{\beta\left(  n+S+1\right)  })\times I^{\omega}\right)  $ which deformation
retracts onto $h_{\psi}(\partial T_{\beta\left(  n+S+1\right)  })\times
I^{\omega}$. Projection yields a deformation retraction of $W$ onto $h_{\psi
}(\partial T_{\beta\left(  n+S+1\right)  })$ and the same contradiction
obtained earlier.
\end{proof}

\section{Some open questions\label{Section: Some open questions}}

The results presented in this paper raise several questions. The most obvious
revolve around the compressibility hypothesis. Roughly speaking,
compressibility allowed us to take advantage of the CAT(0) geometry of the
product spaces $T\times\widetilde{M}_{v}$ and $T\times\left(  \widetilde{M}%
_{v}\times I^{\omega}\right)  $, even when the corresponding proper cocompact
action is not by isometries. In the absence of compressibility, a different
strategy is clearly needed. Nonetheless, the questions remain.\medskip

\noindent\textbf{Question. }\emph{Suppose }$G$\emph{ is the fundamental group
of a finite connected graph of groups }$\left(  \mathcal{G},\Gamma\right)
$\emph{ with the property that each vertex group }$G_{v}$\emph{ is the
fundamental group of a closed aspherical manifold }$M_{v}$\emph{ and, for each
edge }$e$\emph{, the monomorphisms }$G_{e}\overset{\phi_{e}^{-}%
}{\longrightarrow}G_{i(e)}$\emph{ and }$G_{e}\overset{\phi_{e}^{+}%
}{\longrightarrow}G_{t(e)}$\emph{ are of finite index. Does }$G$\emph{ admit a
}$\mathcal{Z}$\emph{-structure? An }$\mathcal{EZ}$\emph{-structure?}\medskip

In attacking the above question, one is likely to bump up against the
unresolved nature of the Borel Conjecture or the Whitehead group conjecture. For that and other reasons, the
following is an appealing special case.\medskip

\noindent\textbf{Question. }\emph{Suppose }$G$\emph{ is the fundamental group
of a finite connected graph of groups }$\left(  \mathcal{G},\Gamma\right)
$\emph{ with the property that each vertex group }$G_{v}$\emph{ is the
fundamental group of a closed, locally CAT(0) manifold
}$M_{v}$\emph{ and, for each edge }$e$\emph{, the monomorphisms }%
$G_{e}\overset{\phi_{e}^{-}}{\longrightarrow}G_{i(e)}$\emph{ and }%
$G_{e}\overset{\phi_{e}^{+}}{\longrightarrow}G_{t(e)}$\emph{ are of finite
index. Does }$G$\emph{ admit a }$\mathcal{Z}$\emph{-structure? An
}$\mathcal{EZ}$\emph{-structure?}\medskip

We expect a positive answer. The point here is that \cite{bl} assures us that the Borel
Conjecture holds and the Whitehead group vanishes for these edge groups. In
addition, we still have CAT(0) geometry to work with.\medskip

If we are going to give up the compressibility hypothesis anyway, Approach II
provides a method for attacking the above questions in even greater
generality. In particular, there may be no need to confine ourselves to closed
aspherical manifolds as vertex groups.\medskip

\noindent\textbf{Question. }\emph{Suppose }$G$\emph{ is the fundamental group
of a finite connected graph of groups }$\left(  \mathcal{G},\Gamma\right)
$\emph{ with the property that each vertex group }$G_{v}$\emph{ is the
fundamental group of a finite aspherical CW-complex }$Y_{v}$\emph{ and, for
each edge }$e$\emph{, the monomorphisms }$G_{e}\overset{\phi_{e}%
^{-}}{\longrightarrow}G_{i(e)}$\emph{ and }$G_{e}\overset{\phi_{e}%
^{+}}{\longrightarrow}G_{t(e)}$\emph{ are of finite index. Does }$G$\emph{
admit a }$\mathcal{Z}$\emph{-structure? An }$\mathcal{EZ}$\emph{-structure?
Does it help to assume that} $\operatorname*{Wh}\left(  G_{e}\right)  =0$
\emph{for all edge groups? that each }$Y_{v}$\emph{ is nonpositively
curved?}\medskip

\noindent The point here is that, with the help of Hilbert cube technology, we
can often conclude that $G$ acts properly and cocompactly on a product space
or even a CAT(0) space---now of the form $T\times(\widetilde{Y}_{v}\times
I^{\omega})$.

\section{\bigskip Appendix: Graphs of covering spaces and actions on products}

In this appendix, we expand upon the notion of a \emph{graph of covering
spaces} as introduced in Section
\ref{Section: Graphs of nonpositively curved n-manifolds}.

A \emph{graph of pointed topological spaces} is a system $\left(
\mathcal{T},\Gamma\right)  $ consisting of:

\begin{enumerate}
\item a connected oriented graph $\Gamma$ with vertex set $E_{0}$ and edge set
$E_{1}$,

\item a collection $\mathcal{T}$ of pointed path-connected topological spaces
$\left(  Y_{s},y_{s}\right)  $ indexed by $E_{0}\cup E_{1}$, and

\item for each $e\in E_{1}$, a pair of continuous \emph{edge maps} $\left(
Y_{i\left(  e\right)  },y_{i\left(  e\right)  }\right)  \overset{p_{e}%
^{-}}{\longleftarrow}\left(  Y_{e},y_{e}\right)  \overset{p_{e}^{+}%
}{\longrightarrow}\left(  Y_{t\left(  e\right)  },y_{t\left(  e\right)
}\right)  $, each inducing a $\pi_{1}$-monomorphism.
\end{enumerate}

\noindent The \emph{total space} of $\left(  \mathcal{T},\Gamma\right)  $,
denoted $\operatorname*{Tot}\left(  \mathcal{T},\Gamma\right)  $, is the
adjunction space
\[
\operatorname*{Tot}\left(  \mathcal{T},\Gamma\right)  =\left(  \bigcup_{v\in
E_{0}}Y_{v}\right)  \cup\left(  \bigcup_{e\in E_{1}}Y_{e}\times\lbrack
0,1]\right)
\]
where $Y_{e}\times\lbrack0,1]$ is glued onto $Y_{o(e)}$ and $Y_{t(e)}$ using
$p_{e}^{-}$ and $p_{e}^{+}$ respectively. There is a natural projection map
$\pi:\operatorname*{Tot}\left(  \mathcal{T},\Gamma\right)  \rightarrow\Gamma$
for which the preimage of each $v\in E_{0}$, is a copy of $Y_{v}$ and for each
point $y$ lying on the interior of an edge $e$, $\pi^{-1}\left(  y\right)  $
is a copy of $Y_{e}$. There is a copy of $\Gamma$ sitting in
$\operatorname*{Tot}\left(  \mathcal{T},\Gamma\right)  $ made up of the images
of $y_{e}\times\left[  -1,1\right]  $ under the quotient map $q$. Under this
realization of $\Gamma$, $\pi$ may be viewed as a retraction. When each
$\left(  Y_{s},y_{s}\right)  $ is a CW-pair and each $p_{e}^{-}$ and
$p_{e}^{+}$ is cellular, $\operatorname*{Tot}\left(  \mathcal{T}%
,\Gamma\right)  $ inherits a natural CW-structure with $\Gamma$ a subcomplex
and $\pi$ a cellular map. Call $\left(  \mathcal{T},\Gamma\right)  $ a
\emph{compact }graph of pointed topological spaces if $\Gamma$ is a finite
graph and each edge and vertex space is compact. This is equivalent to
requiring $\operatorname*{Tot}\left(  \mathcal{T},\Gamma\right)  $ to be compact.

Given a graph of pointed topological spaces $\left(  \mathcal{T}%
,\Gamma\right)  $, there is an \emph{induced graph of groups }$\left(
\mathcal{G},\Gamma\right)  $ with vertex and edge groups $G_{s}=\pi_{1}\left(
Y_{s},y_{s}\right)  $, and edge monomorphisms $\phi_{e}^{-}=\left(  p_{e}%
^{-}\right)  _{\#}$ and $\phi_{e}^{+}=\left(  p_{e}^{+}\right)  _{\#}$.
Moreover, given a graph of groups $\left(  \mathcal{G},\Gamma\right)  $, it is
possible to realize $\left(  \mathcal{G},\Gamma\right)  $ as a graph of
pointed topological spaces $\left(  \mathcal{T},\Gamma\right)  $; if desired
the $\left(  Y_{s},y_{s}\right)  $ can be chosen to be CW-pairs and the maps
to be cellular. Since numerous choices are involved, there is a great deal of
flexibility in choosing a graph of topological spaces realizing a given graph
of groups.

Suppose each $G_{s}$ has presentation
\[
G_{s}=\langle A_{s}\,|\,R_{s}\rangle.
\]

\begin{definition}
\label{d:fundamental group of graph of groups} Given a maximal tree
$\Gamma_{0}\subseteq\Gamma$, the \emph{fundamental group of }$\left(
\mathcal{G},\Gamma\right)  $ \emph{based at }$\Gamma_{0}$ and denoted $\pi
_{1}\left(  \mathcal{G},\Gamma;\Gamma_{0}\right)  $ has generators
\[
\left(  \bigcup_{v\in E_{0}}A_{v}\right)  \cup\{t_{e}\,|\,e\in E_{1}\}
\]
and relations
\[
\bigcup_{v\in R_{v}}R_{v}\cup\{t_{e}^{-1}\phi_{e}^{-}(g)t_{e}=\phi_{e}%
^{+}(g)\,|\,g\in G_{e},e\in E_{1}\}\cup\{t_{e}=1\,|\,e\in\Gamma_{0}\}
\]

\end{definition}

A notable property of $\pi_{1}\left(  \mathcal{G},\Gamma;\Gamma_{0}\right)  $
is that it contains canonical copies of each $G_{v}$, and---up to
isomorphism---it does not depend on the choice of $\Gamma_{0}$ (although the
canonical copies of $G_{v}$ does). Furthermore, if $\left(  \mathcal{T}%
,\Gamma\right)  $ is a corresponding graph of pointed topological spaces,
there is a natural isomorphism between $\pi_{1}\left(  \mathcal{G}%
,\Gamma;\Gamma_{0}\right)  $ and $\pi_{1}\left(  \operatorname*{Tot}\left(
\mathcal{T},\Gamma\right)  ,\Gamma_{0}\right)  $. Here we use the fundamental
group of $\operatorname*{Tot}\left(  \mathcal{T},\Gamma\right)  $ based at
$\Gamma_{0}$ rather than the usual fundamental group based at a point. This is
a matter of convenience; if $v$ is any of the vertices of $\Gamma_{0}$, there
is a natural isomorphism between $\pi_{1}\left(  \operatorname*{Tot}\left(
\mathcal{T},\Gamma\right)  ,v\right)  $ and $\pi_{1}\left(
\operatorname*{Tot}\left(  \mathcal{T},\Gamma\right)  ,\Gamma_{0}\right)  $.
See \cite{geoghegan} for details.

Of particular interest to us are a pair of spaces on which $\pi_{1}\left(
\mathcal{G},\Gamma;\Gamma_{0}\right)  $ act: the Bass-Serre tree for $\left(
\mathcal{G},\Gamma\right)  $, and the universal cover
$\widetilde{\operatorname*{Tot}\left(  \mathcal{T},\Gamma\right)  }$ of a
corresponding total space.

\begin{itemize}
\item \emph{The Bass-Serre tree }$T$ for $\left(  \mathcal{G},\Gamma\right)  $
has vertex set $\widehat{E}_{0}$ containing one element for each left coset of
each $G_{v}\leq\pi_{1}\left(  \mathcal{G},\Gamma;\Gamma_{0}\right)  $ and edge
set $\widehat{E}_{1}$ containing one element for each left coset of each
$G_{e}\leq\pi_{1}\left(  \mathcal{G},\Gamma;\Gamma_{0}\right)  $. The edge
corresponding to a coset $aG_{e}$ connects the two vertices whose
corresponding cosets contain $aG_{e}$. The left action on $T$ is the obvious
one, a key fact being that the stabilizer of a vertex corresponding to a coset
$aG_{v}$ is the group $aG_{v}a^{-1}$ and the stabilizer of the edge
corresponding to a coset $aG_{e}$ is $aG_{e}a^{-1}$. The quotient of this
action is the original graph $\Gamma$; let $q:T\rightarrow\Gamma$ be that
quotient map.

\item \emph{The universal cover} $\widetilde{\operatorname*{Tot}\left(
\mathcal{T},\Gamma\right)  }$, on the other hand, admits a proper and free
$\pi_{1}\left(  \mathcal{G},\Gamma;\Gamma_{0}\right)  $-action (by covering
transformations); it is cocompact if and only if $\left(  \mathcal{T}%
,\Gamma\right)  $ is a compact graph of spaces.
\end{itemize}

The spaces $T$ and $\widetilde{\operatorname*{Tot}\left(  \mathcal{T}%
,\Gamma\right)  }$ and their actions are closely related. The space
$\widetilde{\operatorname*{Tot}\left(  \mathcal{T},\Gamma\right)  }$ can be
viewed as $\operatorname*{Tot}\left(  \mathcal{U},T\right)  $ where

\begin{enumerate}
\item The Bass-Serre tree $T$ is oriented so that $q:T\rightarrow\Gamma$ is
orientation preserving,

\item For each $s\in\widehat{E}_{0}\cup\widehat{E}$ the vertex/edge space is
$\left(  \widetilde{Y}_{s},\widetilde{y}_{s}\right)  $ where $\widetilde{Y}%
_{s}$ is the universal cover of $\widetilde{Y}_{q(s)}$ and $\widetilde{y}_{s}$
is a preimage of $y_{q\left(  s\right)  }$.

\item The edge maps $\left(  \widetilde{Y}_{i\left(  e\right)  }%
,\widetilde{y}_{i\left(  e\right)  }\right)  \overset{\widetilde{p_{e}^{-}%
}}{\longleftarrow}\left(  \widetilde{Y}_{e},\widetilde{y}_{e}\right)
\overset{\widetilde{p_{e}^{+}}}{\longrightarrow}\left(  Y_{t\left(  e\right)
},y_{t\left(  e\right)  }\right)  $ are the (unique) pointed lifts of the edge
maps $\left(  Y_{o\left(  q\left(  e\right)  \right)  },y_{o\left(  q\left(
e\right)  \right)  }\right)  \overset{p_{q\left(  e\right)  }^{-}%
}{\longleftarrow}\left(  Y_{q\left(  e\right)  },y_{q\left(  e\right)
}\right)  \overset{p_{q\left(  e\right)  }^{+}}{\longrightarrow}\left(
Y_{t\left(  q\left(  e\right)  \right)  },y_{t\left(  q\left(  e\right)
\right)  }\right)  $.
\end{enumerate}

\noindent As such, there is a $\pi_{1}\left(  \mathcal{G},\Gamma;\Gamma
_{0}\right)  $-equivariant projection $\pi:\widetilde{\operatorname*{Tot}%
\left(  \mathcal{T},\Gamma\right)  }\rightarrow T$ so that: for each
$v\in\widehat{E}_{0}$ corresponding to coset $aG_{v}$, $\pi^{-1}\left(
v\right)  \approx\widetilde{Y}_{v}$; and for each point $y$ on the interior of
$e\in\widehat{E}_{1}$ corresponding to coset $aG_{e}$, $\pi^{-1}\left(
y\right)  \approx\widetilde{Y}_{e}$. (An alternative construction of the
Bass-Serre tree is as the quotient of $\widetilde{\operatorname*{Tot}\left(
\mathcal{T},\Gamma\right)  }$ obtained by identifying these covering spaces to
points. See \cite{geoghegan}.) The equivariance of $\pi$ means that, for a
vertex $v$ of $T$ stabilized by $aG_{v}a^{-1}$, the set $\pi^{-1}\left(
v\right)  \approx\widetilde{Y}_{v}$ is stabilized (setwise) by $aG_{v}a^{-1}$.

\subsection{Graphs of covering
spaces\label{Subsection: Graphs of covering spaces}}

If each map in a graph of pointed topological spaces $\left(  \mathcal{T}%
,\Gamma\right)  $ is a covering projection, we call $\left(  \mathcal{T}%
,\Gamma\right)  $ a \emph{graph of covering spaces}. By covering space theory,
every group monomorphism can be realized as a covering projection, but
realizing an arbitrary graph of groups $\left(  \mathcal{G},\Gamma\right)  $
as a graph of covering spaces requires compatibility between these
projections. To obtain such a realization, choices are required under which
each edge space $Y_{e}$ \emph{simultaneously} covers $Y_{i\left(  e\right)  }$
and $Y_{t\left(  e\right)  }$ (in a manner that realizes the given group
monomorphisms). When this is possible, we say that $\left(  \mathcal{G}%
,\Gamma\right)  $ is \emph{realizable by covering spaces.}

\begin{example}
Suppose that each vertex and edge group is isomorphic to $\mathbb{Z}$, so that
$\pi_{1}\left(  \mathcal{G},\Gamma;\Gamma_{0}\right)  $ is a \emph{generalized
Baumslag-Solitar group}. By placing a copy of $S^{1}$ at each vertex and
noting that every finite-sheeted cover of $S^{1}$ is homeomorphic to $S^{1}$,
we see that $\left(  \mathcal{G},\Gamma\right)  $ is realizable by compact
graph of finite-sheeted covering spaces.
\end{example}

\begin{example}
As a generalization of the above, place a copy of $\mathbb{Z}^{n}$ on each
vertex and edge of $\Gamma$. For edge maps, choose arbitrary monomorphisms of
varying index. This graph can be realized by a compact graph of finite-sheeted
covering spaces, where the space on each vertex and edge is the $n$-torus
$T^{n}$.
\end{example}

\begin{example}
Suppose $\left(  \mathcal{G},\Gamma\right)  $ is a graph of finitely generated
free groups, and all monomorphisms are finite index. If we restrict ourselves
to graphs as vertex and edge spaces, there are instances where it is
impossible realize $\left(  \mathcal{G},\Gamma\right)  $ with covering spaces.
However, if we allow the use of $3$-dimensional orientable handlebodies, where
genus determines topological type, we can realize any such $\left(
\mathcal{G},\Gamma\right)  $ as a compact graph of finite-sheeted covering
spaces. (Unfortunately, for our purposes, the corresponding universal covers
will not be compressible.)
\end{example}

A key ingredient in our main theorems is the following general fact.

\begin{lemma}
\label{Lemma: Product structure for graphs of covering spaces}Let $\left(
\mathcal{T},\Gamma\right)  $ be a graph of covering spaces and $v_{0}\in
\Gamma$ be a vertex. Then $\widetilde{\operatorname*{Tot}\left(
\mathcal{T},\Gamma\right)  }$ is homeomorphic to $T\times\widetilde{Y}_{v_{0}%
}$, where $T$ is the Bass-Serre tree corresponding to the induced graph of groups.
\end{lemma}

\begin{proof}
Since each map $p_{e}^{-}:Y_{e}\rightarrow Y_{i\left(  e\right)  }$ [resp.,
$p_{e}^{+}:Y_{e}\rightarrow Y_{t\left(  e\right)  }$] is a covering map, so is
each $\widetilde{p_{e}^{-}}:\widetilde{Y}_{e}\rightarrow\widetilde{Y}%
_{i\left(  e\right)  }$ [resp., $\widetilde{p_{e}^{+}}:\widetilde{Y}%
_{e}\rightarrow\widetilde{Y}_{t\left(  e\right)  }$]. By uniqueness of
universal covers, these latter maps are necessarily homeomorphism. Thus, all
of the double mapping cylinders in the construction of
$\widetilde{\operatorname*{Tot}\left(  \mathcal{T},\Gamma\right)  }$ are
actual products. It follows easily that $\pi:\widetilde{\operatorname*{Tot}%
\left(  \mathcal{T},\Gamma\right)  }\rightarrow T$ is a fiber bundle. Since
$T$ is contractible, it is a trivial bundle.
\end{proof}

Next, for graphs of covering spaces, we give a concrete description of the
action of $\pi_{1}\left(  \operatorname*{Tot}\left(  \mathcal{T}%
,\Gamma\right)  ,\Gamma_{0}\right)  $ on $\widetilde{\operatorname*{Tot}%
\left(  \mathcal{T},\Gamma\right)  }$, where the latter is viewed as
$T\times\widetilde{Y}_{v_{0}}$. Specifically, we will define a $\pi_{1}\left(
\operatorname*{Tot}\left(  \mathcal{T},\Gamma\right)  ,\Gamma_{0}\right)
$-action on $\widetilde{Y}_{v_{0}}$ which, when paired diagonally with the
Bass-Serre action on $T$, gives the desired covering space action on
$T\times\widetilde{Y}_{v_{0}}$.

To define the desired action on $\widetilde{Y}_{v_{0}}$ it is enough to:

\begin{itemize}
\item define, for each $v\in E_{0}$, a homomorphism $\theta_{v}:G_{v}%
\rightarrow\operatorname*{Homeo}(\widetilde{Y}_{v_{0}})$,

\item define a homomorphism $\theta_{F}:F\left(  E_{1}\right)  \rightarrow
\operatorname*{Homeo}(\widetilde{Y}_{v_{0}}),$

\item let $\overline{\Theta}:\left(  \underset{_{v\in E_{0}}}{\ast}%
G_{v}\right)  \ast F\left(  E_{1}\right)  \rightarrow\operatorname*{Homeo}%
(\widetilde{Y}_{v_{0}})$ be the union of the above homomorphisms, and

\item check that all relators described in Definition
\ref{d:fundamental group of graph of groups} are sent to $\operatorname*{id}%
_{Y_{v_{0}}}$.
\end{itemize}

Toward that end, inductively orient the edges of $\Gamma_{0}$ outward away
from $v_{0}$. Then, orient each edge not in $\Gamma_{0}$ arbitrarily. By
changing some symbols, but without loss of generality, we may assume this is
the orientation on $\Gamma$ used in the basic definitions. As noted above, for
each $e\in E_{1}$ we have homeomorphisms

\begin{center}
$\begin{CD}
(\widetilde{Y}_{i(e)},\widetilde{y}_{i(e)}) @<\widetilde{p_{e}^{-}}<< (\widetilde{Y}_{e},\widetilde{y}_{e}) @>\widetilde{p_{e}^{+}
}>>(\widetilde{Y}_{t(e)},\widetilde{y}_{t(e)})
\end{CD}$
\end{center}

Let $f_{e}:(\widetilde{Y}_{i(e) },\widetilde{y}_{i(e)}) \rightarrow
(\widetilde{Y}_{t(e) },\widetilde{y}_{t(e) }) $ be the composition
$\widetilde{p_{e}^{+}}\circ( \widetilde{p_{e}^{-}}) ^{-1}$ and for each $v\in
E_{0}$, let $h_{v}=f_{e_{k}}\circ\cdots\circ f_{e_{1}}:Y_{v_{0}}\rightarrow
Y_{v}$ where $e_{1}\ast\cdots\ast e_{k}$ is the reduced edge path in
$\Gamma_{0}$ from $v_{0}$ to $v$. (Let $h_{v_{0}}=\operatorname*{id}%
_{\widetilde{Y}_{v_{0}}}$.) Since basepoints have been chosen, we have a
well-defined $G_{v}$-action on each vertex space $\widetilde{Y}_{v}$. Viewing
each $\alpha\in G_{v}$ as as a self-homeomorphism of $\widetilde{Y}_{v}$,
define $\theta_{v}:G_{v}\rightarrow\operatorname*{Homeo}(\widetilde{Y}_{v_{0}%
})$ by $\theta_{v}\left(  \alpha\right)  =h_{v}^{-1}\alpha h_{v}$.

To define $\theta_{F}:F\left(  E_{1}\right)  \rightarrow\operatorname*{Homeo}%
(\widetilde{Y}_{v_{0}})$ we need only specify the images of the generators. Do
this by setting $\theta_{F}\left(  e\right)  =h_{t\left(  e\right)  }%
^{-1}\circ f_{e}\circ h_{i\left(  e\right)  }$. Note that if $e\in\Gamma_{0}$,
then $f_{e}\circ h_{i\left(  e\right)  }=h_{t\left(  e\right)  }$, so
$\theta_{F}\left(  e\right)  =\operatorname*{id}_{Y_{v_{0}}}$; hence all type
(ii) relators are sent to the identity element. The key to checking that type
(i) relators $r=e\cdot\left(  p_{e}^{-}\right)  _{\#}\left(  \beta\right)
\cdot e^{-1}\cdot\left(  \left(  p_{e}^{+}\right)  _{\#}\left(  \beta\right)
\right)  ^{-1}$ are sent to the identity is the observation that $( p_{e}^{-})
_{\#}(\beta) =\widetilde{p_{e}^{-}}\circ\beta\circ(\widetilde{p_{e}^{-}})
^{-1}$ and $(p_{e}^{+})_{\#}\left(  \beta\right)  =\widetilde{p_{e}^{+}}%
\circ\beta\circ(\widetilde{p_{e}^{+}})^{-1}$.

It suffices to show that $\overline{\Theta}(e\cdot\left(  p_{e}^{-}\right)
_{\#}\left(  \beta\right)  \cdot e^{-1})=\overline{\Theta}\left(  \left(
p_{e}^{+}\right)  _{\#}\left(  \beta\right)  \right)  $. We provide that calculation.%

\[
\overline{\Theta}\left(  e\cdot\left(  p_{e}^{-}\right)  _{\#}\left(
\beta\right)  \cdot e^{-1}\right)  =\theta_{F}\left(  e\right)  \cdot
\theta_{i\left(  e\right)  }(\left(  p_{e}^{-}\right)  _{\#}\left(
\beta\right)  )\cdot\theta_{F}\left(  e\right)  ^{-1}%
\]
\[
=(h_{t\left(  e\right)  }^{-1}\circ f_{e}\circ h_{i\left(  e\right)  }%
)\cdot(h_{i\left(  e\right)  }^{-1}\circ\left(  p_{e}^{-}\right)  _{\#}\left(
\beta\right)  \circ h_{i\left(  e\right)  })\cdot(h_{i\left(  e\right)  }%
^{-1}\circ f_{e}^{-1}\circ h_{t\left(  e\right)  })
\]
\[
=(h_{t\left(  e\right)  }^{-1}\circ f_{e}\circ h_{i\left(  e\right)  }%
)\cdot(h_{i\left(  e\right)  }^{-1}\circ\left(  p_{e}^{-}\right)  _{\#}\left(
\beta\right)  \circ h_{i\left(  e\right)  })\cdot(h_{i\left(  e\right)  }%
^{-1}\circ f_{e}^{-1}\circ h_{t\left(  e\right)  })
\]
\[
=h_{t\left(  e\right)  }^{-1}\circ f_{e}\circ\left(  p_{e}^{-}\right)
_{\#}\left(  \beta\right)  \circ f_{e}^{-1}\circ h_{t\left(  e\right)  }%
\]
\[
=h_{t\left(  e\right)  }^{-1}\circ f_{e}\circ\left(  \widetilde{p_{e}^{-}%
}\circ\beta\circ\left(  \widetilde{p_{e}^{-}}\right)  ^{-1}\right)  \circ
f_{e}^{-1}\circ h_{t\left(  e\right)  }%
\]
\[
=h_{t\left(  e\right)  }^{-1}\circ\left(  \widetilde{p_{e}^{+}}\circ\left(
\widetilde{p_{e}^{-}}\right)  ^{-1}\right)  \circ\left(  \widetilde{p_{e}^{-}%
}\circ\beta\circ\left(  \widetilde{p_{e}^{-}}\right)  ^{-1}\right)
\circ\left(  \widetilde{p_{e}^{-}}\circ\left(  \widetilde{p_{e}^{+}}\right)
^{-1}\right)  \circ h_{t\left(  e\right)  }%
\]
\[
=h_{t\left(  e\right)  }^{-1}\circ\widetilde{p_{e}^{+}}\circ\beta\circ\left(
\widetilde{p_{e}^{+}}\right)  ^{-1}\circ h_{t\left(  e\right)  }%
\]
\[
=h_{t\left(  e\right)  }^{-1}\circ(p_{e}^{+})_{\#}(\beta)\circ h_{t\left(
e\right)  }%
\]
\[
=\theta_{t\left(  e\right)  }(\left(  p_{e}^{+}\right)  _{\#}\left(
\beta\right)  )
\]
\[
=\overline{\Theta}\left(  \left(  p_{e}^{+}\right)  _{\#}\left(  \beta\right)
\right)
\]

Since all relators of the relative presentation of $\pi_{1}\left(
\operatorname*{Tot}\left(  \mathcal{T},\Gamma\right)  ,\Gamma_{0}\right)  $
are sent to the identity element, $\overline{\Theta}$ induces a homomorphism
$\Theta:\pi_{1}\left(  \operatorname*{Tot}\left(  \mathcal{T},\Gamma\right)
,\Gamma_{0}\right)  \rightarrow\operatorname*{Homeo}(\widetilde{Y}_{v_{0}})$.
This is the desired action. As with the Bass-Serre action of $\pi_{1}\left(
\operatorname*{Tot}\left(  \mathcal{T},\Gamma\right)  ,\Gamma_{0}\right)
$-action on $T$, we do not expect the $\pi_{1}\left(  \operatorname*{Tot}%
\left(  \mathcal{T},\Gamma\right)  ,\Gamma_{0}\right)  $-action on
$\widetilde{Y}_{v_{0}}$ to be proper or free. But combined, these two actions
yield a proper free action on $T\times\widetilde{Y}_{v_{0}}$.

\begin{remark}
\label{Remark: Action by quasi-isometric homeomorphisms}In cases where
$\left(  \mathcal{T},\Gamma\right)  $ is a graph of finite-sheeted covering
spaces, all of the lift homeomorphisms $\left(  \widetilde{Y}_{i\left(
e\right)  },\widetilde{y}_{i\left(  e\right)  }\right)
\overset{\widetilde{p_{e}^{-}}}{\longleftarrow}\left(  \widetilde{Y}%
_{e},\widetilde{y}_{e}\right)  \overset{\widetilde{p_{e}^{+}}}{\longrightarrow
}\left(  \widetilde{Y}_{t\left(  e\right)  },\widetilde{y}_{t\left(  e\right)
}\right)  $ are quasi-isometric homeomorphisms. Since the group action
$\Theta:\pi_{1}\left(  \operatorname*{Tot}\left(  \mathcal{T},\Gamma\right)
,\Gamma_{0}\right)  \rightarrow\operatorname*{Homeo}(\widetilde{Y}_{v_{0}})$,
as described above, takes each group element to a finite composition of lift
homeomorphisms, inverses of those homeomorphisms, and isometries of vertex
spaces, the action is by quasi-isometric homeomorphisms.
\end{remark}

\obeylines

\end{document}